\documentclass{article}

\usepackage{graphicx}
\usepackage{amsmath, amssymb, amsthm, enumerate}

\usepackage{amsmath,amssymb,amsthm,enumerate,mathrsfs}
\newtheorem{theorem}{Theorem}[section]

\newtheorem{lemma}[theorem]{Lemma}
\newtheorem{proposition}[theorem]{Proposition}
\theoremstyle{definition}

\newtheorem{assum}{Assumption}[section]

\newtheorem{remark}[theorem]{Remark}
\usepackage{flowchart}
\usepackage{subfig}
\usepackage{algorithm}
\usepackage{algorithmic}

\numberwithin{equation}{section}

\author{Mingyu Mo\thanks{School of Mathematical Sciences, South China Normal University, Guangzhou, Guangdong, China
  (\email{mmymaths@qq.com}).}
\and Yimin Wei\thanks{School of Mathematical Sciences, Fudan University, Shanghai, China
  (\email{ymwei@fudan.edu.cn}).}
\and Qi Ye\thanks{School of Mathematical Sciences, South China Normal University, Guangzhou, Guangdong, China (\email{yeqi@m.scnu.edu.cn}). Pazhou Lab, Guangzhou, Guangdong, China.}
}

\begin{document}
\makeatletter

\begin{center}
\large{\bf Splitting Method for Support Vector Machine with Lower Semi-continuous Loss}
\end{center}\vspace{5mm}
\begin{center}

\textsc{Mingyu Mo, Qi $\text{\normalfont{Ye}}^{*}$}\end{center}

\vspace{2mm}

\footnotesize{
\noindent\begin{minipage}{14cm}
{\bf Abstract:}
   In this paper, we study the splitting method based on alternating direction method of multipliers for support vector machine in reproducing kernel Hilbert space with lower semi-continuous loss function. If the loss function is lower semi-continuous and subanalytic, we use the Kurdyka-Lojasiewicz inequality to show that the iterative sequence induced by the splitting method globally converges to a stationary point. The numerical experiments also demonstrate the effectiveness of the splitting method.
\end{minipage}
 \\[5mm]



\hbox to14cm{\hrulefill}\par


\section{Introduction}
\label{sec:1}
    Support vector machine (SVM) is a well-known model for binary classification in machine learning. The basic idea of SVM is to find a function in a kernel-based function space to achieve the smallest regularized empirical risk such that the classification rule is constructed by the decision function and the sign function. The SVM is already achieved in reproducing kernel Hilbert spaces (RKHS) and convex loss functions (See \cite{MR2450103}). After the success of the SVM in the RKHS with the convex loss functions, people continue to investigate whether the SVM in the RKHS can be feasible for nonconvex loss functions. Recently, different nonconvex loss functions are proposed and used for traditional SVM (see \cite{Brooks2011, Feng2016, Liu2016, Park2011, Perez-Cruz2003, Shen2017, Yang2018}). But the algorithms of the SVM in the RKHS with the general nonconvex loss function are still lack of study. If the loss function is convex, then it is lower semi-continuous (L.S.C.) while the lower semi-continuity can not imply the convexity. Moreover, the SVM in the RKHS with the L.S.C. and nonconvex loss functions can perform better than the convex loss functions in manny cases (see Section 5).  Currently, people are most interested in the infinite-dimensional spaces for applications of machine learning such that the learning algorithm can be chosen from the enough large amounts of suitable solutions. In this paper, we consider the SVM in the infinite-dimensional RKHS with the L.S.C. loss function (See Section \ref{sec:2}).

    To construct the SVM in the RKHS, we will find a minimizer from Optimization (\ref{2.1}). In this paper, we show that Optimization (\ref{2.1}) has a minimizer in a finite-dimensional space spanned by the reproducing kernel basis related to the training data points. Thus Optimization (\ref{2.1}) can be equivalently transferred to a finite-dimensional Optimization. Based the equivalent Optimization, we discuss the splitting method for Optimization (\ref{2.1}) based on alternating direction method of multipliers (ADMM). By the splitting method, we obtain two subproblems which are computable easily. Moreover, the convergence of ADMM is already guaranteed well for the convex Optimizations (see \cite{Boyd2011Distributed}) and some special nonconvex Optimizations by the Kurdyka-Lojasiewicz (KL) property (see \cite{Guo2016Convergence,Li2015Global}). To complete the proof, we reexchange the convergence theorems in \cite{Guo2016Convergence,Li2015Global} and verify the convergence of the splitting method for Optimization (\ref{2.1}) if the loss function is L.S.C. and subanalytic for the global convergence to a stationary point and the error bound.


   This paper is organized as follows. We introduce the notations and preliminary materials of the SVM in Section \ref{sec:2}.  Next, we study the representer theorem for the L.S.C. loss functions and discuss how to solve Optimization (\ref{2.1}) by the splitting method in Section \ref{sec:3}. Moreover, we discuss the global convergence and convergent rates of splitting method for L.S.C. and subanalytic loss function in Section \ref{sec:4}. Finally, we look at the numerical examples of the different loss functions and RKHSs for the synthetic data and the real data in Section \ref{sec:5}.

    \section{Notations and Preliminaries}
   \label{sec:2}
    In this section, we review some notations and preliminaries of SVM. We denote the set of positive integers as $\mathbb{N}$ and the $d$-dimensional Euclidean space as $\mathbb{R}^{d}$ respectively. For the sample space $X\subseteq \mathbb{R}^{d}$ and the label space $Y=\{+1,-1\}$, the training data $D:=\{(\boldsymbol{x}_{i},y_{i}):i=1,2,...,N\}$ is composed of input data $\boldsymbol{x}_{1},\boldsymbol{x}_{2},...,\boldsymbol{x}_{N}\in X$ and output data $y_{1},y_{2},...,y_{N}\in Y$. We will find a mapping $\mathcal{R}:X\to Y$ related to $D$ such that $\mathcal{R}(\boldsymbol{x})$ is a good approximation of the response $y$ to an arbitrary $\boldsymbol{x}$. For the rest of this paper, the symbol $\boldsymbol{x}$ always refers to column vector, so  $\boldsymbol{x}^{T}$ denotes row vector.

    The SVM is an important class of mappings $\mathcal{R}:X\to Y$. The basic idea of traditional SVM is to find a hyperplane in $X$ that classifies all the training data in $D$ correctly and creates the biggest margin. Thus we construct $\mathcal{R}:X\to Y$ by the hyperplane and the sign function. However, the hyperplane to separate $D$ may not exist and we can only accept the hyperplane that misclassifies some training data. To help us define what we mean by "good", we introduce the loss function of the SVM to find the hyperplane and its equivalent optimization is a optimization with an offset term in a RKHS consisting all linear functions on $X$ with the linear reproducing kernel (See \cite[1.3 History of SVMs and Geometrical Interpretation]{MR2450103}). Obviously, the RKHS is a finite-dimensional kernel-based function space which is isometrically isomorphic to $\mathbb{R}^{d}$. Generally speaking, the SVM is to find a function on $X$ that achieves the smallest regularized empirical risk in a given RKHS such that the classification rule $\mathcal{R}:X\to Y$ is constructed by the decision function and the sign function.

    Since the RKHS is only a finite-dimensional space, in the rest of this paper, we shall discuss the SVM in some special infinite-dimensional spaces with different kinds of kernels. Generally speaking, a Hilbert space $\mathcal{H}$ of functions $f:X\rightarrow \mathbb{R}$ equipped with the complete inner product $\langle \cdot,\cdot \rangle_{\mathcal{H}}$ and the reproducing kernel $K:X\times X\rightarrow \mathbb{R}$ is called an RKHS if it satisfies the following two conditions

    (i) $K(\boldsymbol{x},\cdot)\in \mathcal{H}$ for all $\boldsymbol{x}\in X$,

    (ii) $f(\boldsymbol{x})=\langle f, K(\boldsymbol{x},\cdot)\rangle_{\mathcal{H}}$ for all $f\in \mathcal{H}$ and all $\boldsymbol{x}\in X$.

    In particular, for any $f,g\in \mathcal{H}$, the corresponding norm $\|f\|_{\mathcal{H}}=\sqrt{\langle f,f \rangle_{\mathcal{H}}}$ and the corresponding metric $\|f-g\|_{\mathcal{H}}$ are complete. On the other hand, the reproducing kernel $K$ of an RKHS $\mathcal{H}$ is uniquely determined. By \cite[Theorem 10.3 and 10.4]{Wendland2005}, $K$ is symmetric and positive definite (see \cite[Definition 4.15]{MR2450103}). In some cases, $K$ can be symmetric and strictly positive definite, for instance, Gaussian kernels, Mat\'ern kernels and so on. In particular, the value $K(\boldsymbol{x},\boldsymbol{x}')$ can often be interpreted as a measure of dissimilarity between the input values $\boldsymbol{x}$ and $\boldsymbol{x}'$.

   For more flexible kernels such as those of Gaussian kernels, which belong to the most important kernels in practice, the offset term has neither a known theoretical nor an empirical advantage. In addition, the theoretical analysis is often substantially complicated by offset term. Thus, we decide to exclusively consider the SVM in the RKHS without an offset term. Let us fix such an RKHS $\mathcal{H}$ and a real number $\lambda>0$, the SVM in the RKHS finds a minimizer of
    \begin{equation}\label{2.1}
       \inf_{f\in \mathcal{H}}\ \frac{1}{N} \sum_{i=1}^{N}L(\boldsymbol{x}_{i},y_{i},f(\boldsymbol{x}_{i}))+\lambda \|f\|_{\mathcal{H}}^{2}.
    \end{equation}
where $L:X\times Y\times \mathbb{R}\rightarrow [0,\infty)$ is a loss function and $\lambda \|f\|_{\mathcal{H}}^{2}$ is the regularization term used to penalize $f$ with the large RKHS norm. In the following, we will interpret $L(\boldsymbol{x},y,f(\boldsymbol{x}))$ as the loss of predicting $y$ by $f(\boldsymbol{x})$ if $\boldsymbol{x}$ is observed, that is, the smaller the value $L(\boldsymbol{x}, y,f(\boldsymbol{x}))$ is, the better $f(\boldsymbol{x})$ predicts $y$ in the sense of $L$. We denote the minimizer of Optimization (\ref{2.1}) as $f_{D}$. Next we use $f_{D}$ to construct the SVM in the RKHS as follows.
    \begin{equation}\label{2.2}
    \mathcal{R}(\boldsymbol{x})=\begin{cases} +1, & f_{D}(\boldsymbol{x}) \geq0,\\
                                 -1, & f_{D}(\boldsymbol{x})<0.
                   \end{cases}
    \end{equation}
Moreover, the classification rules constructed by different minimizers of Optimization (\ref{2.1}) have no difference in performance. In conclusion, we just need to find a minimizer of Optimization (\ref{2.1}). Next we discuss the solution of Optimization (\ref{2.1}) and splitting method for Optimization (\ref{2.1}) in Section \ref{sec:3}.

    \section{Representer Theorem and Splitting Method}
    \label{sec:3}
    In this section, we discuss how to solve Optimization (\ref{2.1}) by splitting method. First we discuss the solution of Optimization (\ref{2.1}). The loss function is the key factor of the SVM in the RKHS, because it not only determines the sensitivity to the noise of training data but also affects the sparsity of the SVM in the RKHS. For binary classification, the most straightforward loss function is the 0-1 loss function, which is an "ideal" loss function (see \cite{Cortes1995}). However, 0-1 loss function is bounded, nonconvex and L.S.C. but discontinuous. Trying to optimize 0-1 loss function directly leads to a L.S.C. and nonconvex optimization problem which is unable to deal with by traditional optimization theory and algorithm. Therefore, a number of surrogate loss functions are proposed in the literature, such as convex loss function, that is, $L(\boldsymbol{x}, y,\cdot)$ is a convex function for all $\boldsymbol{x}\in X$ and $y\in Y$. Besides convexity, we can define other loss functions in a similar way, such as continuity, smoothness, lower semi-continuity, etc. Specially, if $L$ is convex, then $L$ is L.S.C. But the lower semi-continuity can not imply the convexity.

    Generally speaking, loss functions can be divided into two categories: convex loss functions and nonconvex loss functions. Convex loss functions including Hinge loss, square loss are the most commonly used. If $L$ is a convex loss function, then the classical representer theorem \cite[Theorem 5.5]{MR2450103} assures that Optimization (\ref{2.1}) has a unique minimizer $f_{D}$ such that
    $$
        f_{D}\in \mathrm{span}\{K(\boldsymbol{x}_{1},\cdot),...,K(\boldsymbol{x}_{N},\cdot)\},
    $$
where $\text{\normalfont{span}}\{K(\boldsymbol{x}_{1},\cdot),...,K(\boldsymbol{x}_{N},\cdot)\}$ denotes the set of all finite linear combinations of $\{K(\boldsymbol{x}_{1},\cdot),...,K(\boldsymbol{x}_{N},\cdot)\}$.
A remarkable consequence of the representation above is the fact that $f_{D}$ is contained in a known finite- dimensional space, even if the space $\mathcal{H}$ itself is substantially larger. Thus, the convex loss functions are viewed as highly preferable in many publications because of their computational advantages (unique minimizer, ease-of-use, ability to be efficiently optimized by convex optimization tools, etc.).

    However, the convexity also offer poor approximations to 0-1 loss function. Therefore, different nonconvex loss functions, such as ramp loss, truncated logistic loss, truncated least square loss, truncated least square loss, bi-truncated loss, generalized exponential loss, generalized logistic loss and Sigmoid loss are proposed and used in SVM (see \cite{Brooks2011, Feng2016, Liu2016, Park2011, Perez-Cruz2003, Shen2017, Yang2018}). These loss functions mentioned above and 0-1 loss functions are L.S.C. and nonconvex. Recently, \cite[Corollary 4.1]{Huang2019} generalizes the classical representer theorem to the L.S.C. loss function. Moreover, by the preliminary numerical experiments, we find that the SVM in the RKHS with the L.S.C. and nonconvex loss functions may be better than the convex loss functions (see Section 5). Thus, we will discuss the SVM in the RKHS with the L.S.C. loss functions. Before we show our main result, we need some concepts and properties of $\mathcal{H}$.

    From \cite[Definition 2.1]{Xu2019}, an RKHS $\mathcal{H}$ can be seen as the two-sided reproducing kernel Banach space. On the other hand, the Riesz Representation Theorem assures that the dual space of $\mathcal{H}^{'}$ is isometrically isomorphic to $\mathcal{H}$, that is, $\mathcal{H}^{'}\cong \mathcal{H}$, which ensures that $\mathcal{H}$ is reflexive. Hence, by \cite[Definition 2.2.27]{Dales} and reflexity of $\mathcal{H}$, it shows that the predual space of $\mathcal{H}$ is $\mathcal{H}^{'}$. Moreover, since $\mathcal{H}$ is also strictly convex and smooth, we obtain the following formula of Fr\'echet derivative (see \cite[Remark 2.24]{Xu2019})
    \begin{equation}\label{3.1}
       \nabla(\|\cdot\|_{\mathcal{H}})(f)=\frac{f}{\|f\|_{\mathcal{H}}}\in \mathcal{H},\ \ \forall f\neq 0,
    \end{equation}
where $\nabla$ denotes the Fr\'echet derivative. In particular, for a differentiable real function, $\nabla$ represents the gradient.

    \begin{proposition}\label{theorem:3.1}
       If $L$ is lower semi-continuous, then Optimization \text{\normalfont{(\ref{2.1})}} has a minimizer $f_{D}$ such that
    $$
      f_{D}\in \mathrm{span}\{K(\boldsymbol{x}_{1},\cdot),...,K(\boldsymbol{x}_{N},\cdot)\}.
    $$
    \end{proposition}

    \begin{proof}
      Since $\mathcal{H}$ is an reproducing kernel Banach space which has a predual space $\mathcal{H}^{'}$, from the condition (i) in the definition of the RKHS, we have that
    \begin{equation*}
      \{K(\boldsymbol{x},\cdot):\boldsymbol{x}\in X\}\subseteq \mathcal{H}\cong \mathcal{H}^{'}.
    \end{equation*}
Therefore, \cite[Corollary 4.1]{Huang2019} ensures that Optimization (\ref{2.1}) has a minimizer $f_{D}$ such that
     \begin{equation*}
       \nabla(\|\cdot\|_{\mathcal{H}})(f_{D}) \in \text{\normalfont{span}}\{K(\boldsymbol{x}_{1},\cdot),...,K(\boldsymbol{x}_{N},\cdot)\}.
     \end{equation*}
If $f_{D}\neq 0$, then (\ref{3.1}) shows that
     \begin{equation*}
        \frac{f_{D}}{\|f_{D}\|_{\mathcal{H}}}\in \text{\normalfont{span}}\{K(\boldsymbol{x}_{1},\cdot),...,K(\boldsymbol{x}_{N},\cdot)\}.
     \end{equation*}
Moreover, by the definition of $\text{\normalfont{span}}\{K(\boldsymbol{x}_{1},\cdot),...,K(\boldsymbol{x}_{N},\cdot)\}$, we have that
     \begin{equation*}
       f_{D}\in \text{\normalfont{span}}\{K(\boldsymbol{x}_{1},\cdot),...,K(\boldsymbol{x}_{N},\cdot)\}.
     \end{equation*}
If $f_{D}=0$, then $0 \in \text{\normalfont{span}}\{K(\boldsymbol{x}_{1},\cdot),...,K(\boldsymbol{x}_{N},\cdot)\}$. In conclusion,
     \begin{equation*}
       f_{D}\in \text{\normalfont{span}}\{K(\boldsymbol{x}_{1},\cdot),...,K(\boldsymbol{x}_{N},\cdot)\}.
     \end{equation*}
The proof is complete.
   \end{proof}

   \begin{remark}
      If $L$ is L.S.C. and nonconvex, then Optimization \text{\normalfont{(\ref{2.1})}} may have more than one minimizers. Moreover, Proposition \text{\normalfont{3.1}} guarantees that at least one of minimizers is contained in $\text{\normalfont{span}}\{K(\boldsymbol{x}_{1},\cdot),...,K(\boldsymbol{x}_{N},\cdot)\}$. Thus, we focus on find the minimizer in $\text{\normalfont{span}}\{K(\boldsymbol{x}_{1},\cdot),...,K(\boldsymbol{x}_{N},\cdot)\}$.
   \end{remark}

    By Proposition \ref{theorem:3.1}, Optimization (\ref{2.1}) in the RKHS $\mathcal{H}$ can be equivalently transferred to Optimization (\ref{2.1}) in $\mathrm{span}\{K(\boldsymbol{x}_{1},\cdot),...,K(\boldsymbol{x}_{N},\cdot)\}$. Next we show that Optimization (\ref{2.1}) in $\mathrm{span}\{K(\boldsymbol{x}_{1},\cdot),...,K(\boldsymbol{x}_{N},\cdot)\}$ also can be equivalently transferred to a finite-dimensional Optimization in $\mathbb{R}^{N}$. We denote the kernel matrix of $K$ as
    \begin{equation*}
       A:=\left(K(\boldsymbol{x}_{i},\boldsymbol{x}_{j})\right)_{i,j=1}^{N,N},
    \end{equation*}
and $A$ is symmetric and positive definite. For each $f \in \text{\normalfont{span}}\{K(\boldsymbol{x}_{1},\cdot),...,K(\boldsymbol{x}_{N},\cdot)\}$, there exists a unique vector $\boldsymbol{c}=(c_{1},...,c_{N})^{T}\in \mathbb{R}^{N}$ such that $f$ has the finite representation
     \begin{equation*}
        f=\sum_{j=1}^{N} c_{j}K(\boldsymbol{x}_{j},\cdot),
     \end{equation*}
which ensures that
     \begin{equation*}
        f(\boldsymbol{x}_{i})=\sum_{j=1}^{N} c_{j}K(\boldsymbol{x}_{j},\boldsymbol{x}_{i})=\sum_{j=1}^{N} K(\boldsymbol{x}_{i},\boldsymbol{x}_{j})c_{j}=(A\boldsymbol{c})_{i}.
     \end{equation*}
Therefore, it is easy to check that
     \begin{equation}\label{3.2}
        f\mapsto (f(\boldsymbol{x}_{1}),f(\boldsymbol{x}_{2}),...,f(\boldsymbol{x}_{N}))^{T}=A\boldsymbol{c}
     \end{equation}
is an isometric isomorphism from $\mathrm{span}\{K(\boldsymbol{x}_{1},\cdot),...,K(\boldsymbol{x}_{N},\cdot)\}$
onto $\mathbb{R}^{N}$. On the other hand, combining the property of inner product $\langle \cdot,\cdot \rangle_{\mathcal{H}}$ with the reproducing property in the RKHS, it holds that
     \begin{equation}\label{3.3}
       \|f\|_{\mathcal{H}}^{2}=\langle f,\sum_{i=1}^{N} c_{i}K(\boldsymbol{x}_{i},\cdot)\rangle_{\mathcal{H}}=\sum_{i=1}^{N} c_{i}\langle f,K(\boldsymbol{x}_{i},\cdot)\rangle_{\mathcal{H}}=\sum_{i=1}^{N} c_{i}f(\boldsymbol{x}_{i})=\boldsymbol{c}^{T}A\boldsymbol{c}.
     \end{equation}
So Optimization (\ref{2.1}) can be equivalently transferred to the following Optimization in $\mathbb{R}^{N}$
     \begin{equation}\label{3.4}
        \min_{\boldsymbol{c}\in \mathbb{R}^{N}}\enskip \frac{1}{N}\sum_{i=1}^N L(\boldsymbol{x}_{i},y_{i},(A\boldsymbol{c})_{i})+\lambda \boldsymbol{c}^{T}A\boldsymbol{c},
     \end{equation}
We denote the minimizer of Optimization (\ref{3.4}) as $\boldsymbol{c}_{D}=((c_{1})_{D},(c_{2})_{D},...,(c_{N})_{D})^{T}$, it follows that
     \begin{equation*}
       f_{D}=\sum_{i=1}^{N}(c_{i})_{D} K(\boldsymbol{x}_{i},\cdot).
     \end{equation*}
This ensures that we employ the finite suitable parameters to reconstruct the SVM in the RKHS.

    By this idea, we consider to find an algorithm based on Optimization (\ref{3.4}) to compute Optimization (\ref{2.1}) easily. At present, we mainly use the subdifferential, proximal operator and Fenchel conjugate of loss function to design algorithms for SVM, such as subgradient method, Lagrange multipliers method and sequential minimal optimization (SMO). These classical numerical algorithms are suitable for solving convex and smooth programs. Since $\boldsymbol{c}\mapsto (A\boldsymbol{c})_{i}$ is continuous, the lower semi-continuity of $L(\boldsymbol{x}_{i},y_{i},\cdot)$ guarantees the lower semi-continuity of $\boldsymbol{c}\mapsto L(\boldsymbol{x}_{i},y_{i},(A\boldsymbol{c})_{i})$ for all $i=1,2,...,N$ which ensures that $F$ is L.S.C.. On the other hand, it is clear that $G$ is continuous. In conclusion, Optimization (\ref{3.4}) is L.S.C. finite-dimensional Optimization which may be nonsmooth or nonconvex. Many classical numerical algorithms are not suitable for Optimization (\ref{3.4}) when $L$ is L.S.C.

    The ADMM algorithm, as one of splitting techniques, has been successfully exploited in a wide range of structured regularization optimization problems in machine learning. ADMM can even be used to minimize nonsmooth or nonconvex function, which solves optimization problem by breaking them into smaller pieces. Moreover, Paper \cite{Wang2021} discusses how to use ADMM for the traditional SVM with the 0-1 loss function. For the general L.S.C. functions and kernels, we observe that the subproblems in ADMM for Optimization (\ref{3.4}) can be transferred into some Optimizations in $\mathbb{R}$ and a well-posed linear system, each of which are thus easier to handle. In summary, we discuss the splitting method based on ADMM for Optimization (\ref{2.1}). Moreover, if the sample size $N$ is small to moderate, then the preliminary numerical experiments show that the splitting method is a fast algorithm for Optimization (\ref{2.1}) (see Section \ref{sec:5}). Hence, we will study how to solve Optimization (\ref{2.1}) by the splitting method. For notational convenience, let
    $$
       F(A\boldsymbol{c})=\frac{1}{N}\sum\limits_{i=1}^N L(\boldsymbol{x}_{i},y_{i},(A\boldsymbol{c})_{i}), \ \  G(\boldsymbol{c})=\lambda \boldsymbol{c}^{T}A\boldsymbol{c}.
    $$
To describe the algorithm, we first reformulate Optimization (\ref{3.4}) as
    \begin{equation}\label{3.5}
    \begin{array}{cl}
    \displaystyle\min_{\boldsymbol a, \boldsymbol c}\enskip & \displaystyle  F(\boldsymbol{\alpha})+G(\boldsymbol{c}),  \\
    \text{s.t.}\enskip & \boldsymbol{\alpha}=A\boldsymbol{c}.
    \end{array}
    \end{equation}
Recall that the augmented Lagrangian function is defined as:
    \begin{equation*}
    \mathcal{L}_{\rho}(\boldsymbol{\alpha},\boldsymbol{c},\boldsymbol{\gamma})=F(\boldsymbol{\alpha})+G(\boldsymbol{c})+\boldsymbol{\gamma}^{T}(\boldsymbol{\alpha}-A\boldsymbol{c})+\frac{\rho}{2}\|\boldsymbol{\alpha}-A\boldsymbol{c}\|^{2},
    \end{equation*}
where the Lagrangian multiplier $\rho>0$ and $\|\cdot\|$ denotes $2$-norm in Euclidean space. The splitting method  is presented as follows. Suppose that the algorithm is initialized at $(\boldsymbol{\alpha}^{0},\boldsymbol{c}^{0}, \boldsymbol{\gamma}^{0})$, its iterative scheme is
      \begin{align}
         \boldsymbol{\alpha}^{k+1}&\in \mathop{\mathrm{argmin}}_{\boldsymbol{\alpha}\in \mathbb{R}^{N}}\ \mathcal{L}_{\rho}(\boldsymbol{\alpha},\boldsymbol{c}^{k},\boldsymbol{\gamma}^{k}),\tag{S-1}  \label{S-1} \\
         \boldsymbol{c}^{k+1}&\in \mathop{\mathrm{argmin}}_{\boldsymbol{c}\in \mathbb{R}^{N}}\ \mathcal{L}_{\rho}(\boldsymbol{\alpha}^{k+1},\boldsymbol{c},\boldsymbol{\gamma}^{k}), \tag{S-2} \label{S-2} \\
         \boldsymbol{\gamma}^{k+1}&:=\boldsymbol{\gamma}^{k}+\rho (\boldsymbol{\alpha}^{k+1}-A\boldsymbol{c}^{k+1}), \label{S-3} \tag{S-3} \\
         s^{k+1}&:=\sum_{i=1}^{N}(\boldsymbol{c}^{k+1})_{i}K(\boldsymbol{x}_{i},\cdot), \label{S-4} \tag{S-4}
       \end{align}
where $k$ is an iteration counter. Since (\ref{S-1}) only depends on $\boldsymbol{\alpha}$ and (\ref{S-2}) only depends on $\boldsymbol{c}$, by combining the linear and quadratic terms in $\mathcal{L}_{\rho}$, we equivalently transfer (\ref{S-1}) and (\ref{S-2}) to
      \begin{align}
         \boldsymbol{\alpha}^{k+1}&\in \mathop{\mathrm{argmin}}_{\boldsymbol{\alpha}\in \mathbb{R}^{N}}\ F(\boldsymbol{\alpha})+\frac{\rho}{2}\|\boldsymbol{\alpha}-A\boldsymbol{c}^{k}+\frac{1}{\rho}\boldsymbol{\gamma}^{k}\|^{2},\tag{S-1'}  \label{S-1'} \\
         \boldsymbol{c}^{k+1}&\in \mathop{\mathrm{argmin}}_{\boldsymbol{c}\in \mathbb{R}^{N}}\ G(\boldsymbol{c})+\frac{\rho}{2}\| \boldsymbol{\alpha}^{k+1}-A\boldsymbol{c}+\frac{1}{\rho}\boldsymbol{\gamma}^{k}\|^{2} \tag{S-2'} \label{S-2'}.
      \end{align}
By definition, it is easy to check that (\ref{S-1'}) is L.S.C. and (\ref{S-2'}) is continuous. Moreover, (\ref{S-1'}) and (\ref{S-2'}) is coercive (see \cite[Definition 2.13]{Beck2017}), that is,
   \begin{align*}
    &\lim_{\|\boldsymbol{\alpha}\|\rightarrow \infty} F(\boldsymbol{\alpha})+\frac{\rho}{2}\|\boldsymbol{\alpha}-A\boldsymbol{c}^{k}+\frac{1}{\rho}\boldsymbol{\gamma}^{k}\|^{2} \to \infty,\\
    &\lim_{\|\boldsymbol{c}\|\rightarrow \infty} G(\boldsymbol{c})+\frac{\rho}{2}\| \boldsymbol{\alpha}^{k+1}-A\boldsymbol{c}+\frac{1}{\rho}\boldsymbol{\gamma}^{k}\|^{2}\to \infty.
   \end{align*}
Thus, Weierstrass Theorem \cite[Theorem 2.14]{Beck2017} assures that (\ref{S-1'}) and (\ref{S-2'}) both have a solution. As a consequence, the splitting method above is well-defined and an infinite iterative sequence $\{(\boldsymbol{\alpha}^{k},\boldsymbol{c}^{k},\boldsymbol{\gamma}^{k},s^{k})\}$ is generated. Moreover, $\{s^{k}\}$ can be seen as an infinite iterative sequence to approximate the minimizer of Optimization (\ref{2.1}).

     Next, we discuss how to solve subproblems (\ref{S-1'}) and (\ref{S-2'}). As for (\ref{S-1'}), since $F$ and $\|\cdot\|^{2}$ can be split of the variable into subvectors, that is,
   \begin{equation*}
      \frac{1}{N}\sum\limits_{i=1}^N L(\boldsymbol{x}_{i},y_{i},\alpha_{i})+\frac{\rho}{2}\|\boldsymbol{\alpha}-A\boldsymbol{c}^{k}+\frac{1}{\rho}\boldsymbol{\gamma}^{k}\|^{2}=\sum_{i=1}^{N} \frac{L(\boldsymbol{x}_{i},y_{i},\alpha_{i})}{N}+\frac{\rho}{2}\left(\alpha_{i}-(A\boldsymbol{c}^{k})_{i}+\frac{1}{\rho}(\boldsymbol{\gamma}^{k})_{i}\right)^{2},
   \end{equation*}
we equivalently transfer an Optimization in $\mathbb{R}^{N}$ to some Optimizations in $\mathbb{R}$, that is,
    \begin{equation}\tag{S-1''}\label{S-1''}
       (\boldsymbol{\alpha}^{k+1})_{i}\in \mathop{\mathrm{argmin}}_{\alpha_{i}}\ \frac{L(\boldsymbol{x}_{i},y_{i},\alpha_{i})}{N} +\frac{\rho}{2}\left(\alpha_{i}-(A\boldsymbol{c}^{k})_{i}+\frac{1}{\rho}(\boldsymbol{\gamma}^{k})_{i}\right)^{2},\ i=1,2,...,N.
    \end{equation}
In other words, we solve (\ref{S-1'}) in $N$-dimensional space by breaking it into $N$ Optimizations (\ref{S-1''}) in 1-dimensional space and each of them is easier to handle. To illustrate above, we give a simple example. Assume that $L$ is Hinge loss, that is,
    $$
        L_{\text{Hinge}}(\boldsymbol{x}_{i},y_{i},\alpha_{i})=\max\{0, 1-y_{i}\alpha_{i}\}.
    $$
By simple algebra, if $y_{i}=+1$, then
    $$
       (\boldsymbol{\alpha}^{k+1})_{i}=
       \begin{cases}
           u_{i}+\frac{1}{\rho N} & u_{i}<1-\frac{1}{\rho N}, \\
           1 & 1-\frac{1}{\rho N}\leq u_{i}<1, \\
           u_{i} & u_{i}\geq 1,
       \end{cases}
    $$
where $u_{i}=(A\boldsymbol{c}^{k})_{i}-\frac{1}{\rho}(\boldsymbol{\gamma}^{k})_{i}$. On the other hand, if $y_{i}=-1$, then
    $$
       (\boldsymbol{\alpha}^{k+1})_{i}=
       \begin{cases}
           u_{i} & u_{i}<-1, \\
           -1 & -1\leq u_{i}<-1+\frac{1}{\rho N}, \\
           u_{i}-\frac{1}{\rho N} & u_{i}\geq -1+\frac{1}{\rho N}.
       \end{cases}
    $$
For the general L.S.C. loss function $L$, the solution set of (\ref{S-1''}) may not be a singleton. For example, assume that $L$ is ramp loss, that is,
    $$
        L_{\text{Ramp}}(\boldsymbol{x}_{i},y_{i},\alpha_{i})=
        \begin{cases}
           1 & 1-y_{i}\alpha_{i}>1, \\
           1-y_{i}\alpha_{i} & 0< 1-y_{i}\alpha_{i}\leq 1, \\
           0 & 1-y_{i}\alpha_{i}<0.
       \end{cases}
    $$
Similarly, by simple algebra, if $0<\frac{1}{\rho N}<2$ and $y_{i}=+1$, then
    $$
       (\boldsymbol{\alpha}^{k+1})_{i}=
       \begin{cases}
           u_{i} & u_{i}\leq -\frac{1}{2\rho N}, \\
           u_{i}\ \text{or}\ u_{i}+\frac{1}{\rho N} & u_{i}=-\frac{1}{2\rho N}, \\
           u_{i}+\frac{1}{\rho N} & -\frac{1}{2\rho N}<u_{i}\leq 1-\frac{1}{pN}, \\
           1 & 1-\frac{1}{\rho N}<u_{i}<1, \\
           u_{i} & u_{i}\geq 1.
       \end{cases}
    $$
if $0<\frac{1}{\rho N}<2$ and $y_{i}=-1$, then
    $$
       (\boldsymbol{\alpha}^{k+1})_{i}=
       \begin{cases}
           u_{i} & u_{i}\leq -1, \\
           -1 & -1<u_{i}<-1+\frac{1}{\rho N}, \\
           u_{i}-\frac{1}{\rho N} & -1+\frac{1}{pN}\leq u_{i}<\frac{1}{2\rho N}, \\
           u_{i}\ \text{or}\ u_{i}-\frac{1}{\rho N} & u_{i}=\frac{1}{2\rho N}, \\
           u_{i} & u_{i}>\frac{1}{2\rho N}.
       \end{cases}
    $$
In this case, we choose one of the elements in the solution set as $(\boldsymbol{\alpha}^{k+1})_{i},\ i=1,...,N$.

    As for (\ref{S-2'}), since $A$ is symmetric and strictly positive definite and $\lambda,\rho>0$, it is clear that (\ref{S-2'}) is nonnegative, convex and continuously differentiable. Moreover, $\boldsymbol{c}^{k+1}$ is a minimizer of (\ref{S-2'}) and thus a stationary point, that is,
    \begin{equation*}
      2\lambda A \boldsymbol{c}^{k+1}-\rho A(\boldsymbol{\alpha}^{k+1}-A\boldsymbol{c}^{k+1}+\frac{1}{\rho}\boldsymbol{\gamma}^{k})=0.
    \end{equation*}
By rearranging terms, $\boldsymbol{c}^{k+1}$ is the solution of the following well-posed linear system
    \begin{equation}\label{S-2''} \tag{S-2''}
     (2\lambda I+\rho A)\boldsymbol{c}=\rho \boldsymbol{\alpha}^{k+1}+\boldsymbol{\gamma}^{k}.
    \end{equation}
where $I$ is the identity matrix with and order $N$. Since $2\lambda I+\rho A$ is symmetric and strictly positive definite, we can use conjugate gradient method to obtain $\boldsymbol{c}^{k+1}$ (See \cite[4.7.3 Conjugate Gradient Methods]{Kincaid2002}).

    When $\boldsymbol{\alpha}^{k+1}$ and $\boldsymbol{c}^{k+1}$ is acquired, we can obtain $\boldsymbol{\gamma}^{k+1}$ by (\ref{S-3}). However, we have a simpler one in mind that accomplishes the same goal. Substituting (\ref{S-3}) into (\ref{S-2''}) and rearranging terms, we have that
    \begin{equation}\label{S-3'} \tag{S-3'}
       \boldsymbol{\gamma}^{k+1}=2\lambda\boldsymbol{c}^{k+1}.
    \end{equation}

    In conclusion, the splitting method of Optimization (\ref{2.1}) can be represented as follows:

    \begin{algorithm}[H]
    \footnotesize
    \caption{Splitting Method for SVM in RKHS with L.S.C Loss Function}
    \begin{algorithmic}
      \STATE \textbf{Input:} initial value $(\boldsymbol{\alpha}^{0},\boldsymbol{c}^{0}, \boldsymbol{\gamma}^{0})$, the training data $D$, loss function $L$, kernel matrix $A$, regularization parameter $\lambda>0$, Lagrange multiplier $\rho>0$ and stopping threshold $\varepsilon_{0}>0$.

      \STATE \textbf{for} $k=0,1,2,...$ \textbf{do}

      \qquad (1) Choose $\boldsymbol{\alpha}^{k+1}$ in $\mathop{\mathrm{argmin}}\limits_{\alpha_{i}}\ \frac{1}{N} L(\boldsymbol{x}_{i},y_{i},\alpha_{i})+\frac{\rho}{2}\| \alpha_{i}-(A\boldsymbol{c}^{k})_{i}+\frac{1}{\rho}(\boldsymbol{\gamma}^{k})_{i}\|^{2},\ i=1,...,N$.

      \qquad (2) Let $\boldsymbol{c}^{k(0)}=\boldsymbol{c}^{k}$, $\boldsymbol{r}^{k(0)}=\rho \boldsymbol{\alpha}^{k+1}+\boldsymbol{\gamma}^{k}-A\boldsymbol{c}^{k}$, $\boldsymbol{d}^{k(0)}=\rho \boldsymbol{\alpha}^{k+1}+\boldsymbol{\gamma}^{k}-A\boldsymbol{c}^{k}$.

      \qquad  \textbf{for} $j=0,1,2,...$ \textbf{do}

      \qquad \qquad (2-1) Set $\boldsymbol{c}^{k(j+1)}\leftarrow \boldsymbol{c}^{k(j)}+\frac{\|\boldsymbol{r}^{k(j)}\|^{2}}{(\boldsymbol{d}^{k(j)})^{T}A\boldsymbol{d}^{k(j)}}\boldsymbol{d}^{k(j)}$.

      \qquad \qquad (2-2) Set $\boldsymbol{r}^{k(j+1)}\leftarrow \boldsymbol{r}^{k(j)}-\frac{\|\boldsymbol{r}^{k(j)}\|^{2}}{(\boldsymbol{d}^{k(j)})^{T}A\boldsymbol{d}^{k(j)}}A\boldsymbol{d}^{k(j)}$.

      \qquad \qquad (2-3) Set $\boldsymbol{d}^{k(j+1)}\leftarrow \boldsymbol{r}^{k(j+1)}+\frac{\|\boldsymbol{r}^{k(j+1)}\|^{2}}{\|\boldsymbol{r}^{k(j)}\|^{2}}\boldsymbol{d}^{k(j)}$.

      \qquad \qquad \textbf{if} $\boldsymbol{r}^{k(j+1)}=\boldsymbol{0}$ \textbf{then} stop.

      \qquad \textbf{end for}

      \qquad \textbf{Output:} The approximate solution $\boldsymbol{c}^{k(j+1)}$ as $\boldsymbol{c}^{k+1}$.

      \qquad (3) Set $\boldsymbol{\gamma}^{k+1}\leftarrow 2\lambda\boldsymbol{c}^{k+1}$.

      \qquad (4) Set $s^{k+1}\leftarrow \sum\limits_{i=1}^{N}(\boldsymbol{c}^{k+1})_{i}K(\boldsymbol{x}_{i},\cdot)$.

      \qquad \textbf{if} $\| \boldsymbol{\alpha}^{k+1}-A\boldsymbol{c}^{k+1}\|<\varepsilon_{0}$ \textbf{then} stop.

      \STATE \textbf{end for}

      \STATE \textbf{Output:} The approximate solution $s^{k+1}$.
    \end{algorithmic}
    \label{alg:ADMM}
    \end{algorithm}

    In Section \ref{sec:4}, we verify that if $L$ is L.S.C. and subanalytic and $\rho$ is sufficiently large, $\{s^{k}\}$ globally converges to a stationary point of Optimization (\ref{2.1}). In particular, if $L$ is convex, then $\{s^{k}\}$ is globally convergent to the minimizer $f_{D}$. If $L$ is L.S.C. and nonconvex, then $\{s^{k}\}$ may converge to a stationary point not a minimizer. Hence, it is better to solve Optimization (\ref{2.1}) repeatedly by selecting some initial values randomly and choosing the minimizer of these outputs as the approximate solution $s_{D}$ of Optimization (\ref{2.1}). Finally, we construct the classification rule by $s_{D}$, that is,
    $$
    \mathcal{R}(\boldsymbol{x})=\begin{cases} +1, & s_{D}(\boldsymbol{x}) \geq0,\\
                                 -1, & s_{D}(\boldsymbol{x})<0.
                   \end{cases}
    $$

    We will complete the convergence analysis of Algorithm \ref{alg:ADMM} in Section \ref{sec:4}.

    \section{Convergence Analysis}
    \label{sec:4}
    In this section, we discuss the convergence of $\{s^{k}\}$ inspired from the work \cite{Guo2016Convergence,Li2015Global} and use similar line of arguments therein. To ensure the convergence, we need the following assumption of loss function.
    \begin{assum}[Loss Function]
    \label{Assumption:4.1}
       For any $\boldsymbol{x}\in X$ and $y\in Y$,

       \text{\normalfont{(i)}} $L(\boldsymbol{x},y,\cdot)$ is lower semi-continuous and subanalytic on $\mathbb{R}$.

      \text{\normalfont{(ii)}} $0$ is not the stationary point of $L(\boldsymbol{x},y,\cdot)$  in the sense of limiting subdifferential, that is, $0\notin \partial L(\boldsymbol{x},y,0)$. \text{\normalfont{(see \cite[Definition 8.3]{Rockafellar1998})}}
    \end{assum}

    Subanalytic functions are quite wide, including semi-algebraic, analytic and semi-analytic functions (see \cite[6.6 Analytic Problems]{Finite}). More precisely, polynomial functions and piecewise polynomial functions are subanalytic functions. However, subanalyticity does not even imply continuity. Moreover, some margin-based loss functions (see \cite[2.3 Margin-Based Losses for Classification Problems]{MR2450103}) satisfy Assumption \ref{Assumption:4.1}, such as the least square loss, the Hinge loss, the truncated least squares loss, logistic loss and so on.

    \begin{theorem}\label{Theorem:4.1}
        Suppose that Assumption \text{\normalfont{\ref{Assumption:4.1}}} holds and Algorithm \text{\normalfont{\ref{alg:ADMM}}} is initialized at $(\boldsymbol{\alpha}^{0},\boldsymbol{c}^{0},\boldsymbol{\gamma}^{0})$. If $A$ is symmetric and strictly positive definite and $\rho>4\lambda\|A^{-1}\|$, then $\{s^{k}\}$ converges to a stationary point $s^{*}$ of Optimization \text{\normalfont{(\ref{2.1})}}.
    \end{theorem}

    Before we verify our main result, we need two lemmas about $\{\mathcal{L}_{\rho}(\boldsymbol{\alpha}^{k},\boldsymbol{c}^{k},\boldsymbol{\gamma}^{k})\}$.

    \begin{lemma}\label{Lemma:4.2}
       If the conditions in Theorem \text{\normalfont{\ref{Theorem:4.1}}} holds, then there exists $\delta>0$ such that the following descent inequality holds
    $$
       \delta\|s^{k+1}-s^{k}\|_{\mathcal{H}}^{2}\leq \mathcal{L}_{\rho}(\boldsymbol{\alpha}^{k},\boldsymbol{c}^{k},\boldsymbol{\gamma}^{k})-\mathcal{L}_{\rho}(\boldsymbol{\alpha}^{k+1},\boldsymbol{c}^{k+1},\boldsymbol{\gamma}^{k+1}).
    $$
    \end{lemma}

    \begin{proof}
       From (\ref{S-1}), we know that $\boldsymbol{\alpha}^{k+1}$ is the minimizer of $\mathcal{L}_{\rho}(\boldsymbol{\alpha},\boldsymbol{c}^{k},\boldsymbol{\gamma}^{k})$, that is,
     \begin{equation}\label{4.1}
        0\leq \mathcal{L}_{\rho}(\boldsymbol{\alpha}^{k},\boldsymbol{c}^{k},\boldsymbol{\gamma}^{k})-\mathcal{L}_{\rho}(\boldsymbol{\alpha}^{k+1},\boldsymbol{c}^{k},\boldsymbol{\gamma}^{k}).
     \end{equation}
Similarily from the definition of $\mathcal{L}_{\rho}$, we see that
     \begin{align*}
       &\mathcal{L}_{\rho}(\boldsymbol{\alpha}^{k+1},\boldsymbol{c}^{k},\boldsymbol{\gamma}^{k})-\mathcal{L}_{\rho}(\boldsymbol{\alpha}^{k+1},\boldsymbol{c}^{k+1},\boldsymbol{\gamma}^{k}) \\
      =&G(\boldsymbol{c}^{k})-G(\boldsymbol{c}^{k+1})+(\boldsymbol{\gamma}^{k})^{T}A(\boldsymbol{c}^{k+1}-\boldsymbol{c}^{k})+\frac{\rho}{2}\|\boldsymbol{\alpha}^{k+1}-A\boldsymbol{c}^{k}\|^{2}-\frac{\rho}{2}\|\boldsymbol{\alpha}^{k+1}-A\boldsymbol{c}^{k+1}\|^{2} \\ \notag
      =&G(\boldsymbol{c}^{k})-G(\boldsymbol{c}^{k+1})+(\boldsymbol{\gamma}^{k})^{T}A(\boldsymbol{c}^{k+1}-\boldsymbol{c}^{k})+\frac{\rho}{2}\|\boldsymbol{\alpha}^{k+1}-A\boldsymbol{c}^{k+1}\|^{2}+\frac{\rho}{2}\|A(\boldsymbol{c}^{k+1}-\boldsymbol{c}^{k})\|^{2}+ \\ \notag
       &\rho(\boldsymbol{\alpha}^{k+1}-A\boldsymbol{c}^{k+1})^{T}A(\boldsymbol{c}^{k+1}-\boldsymbol{c}^{k})-\frac{\rho}{2}\|\boldsymbol{\alpha}^{k+1}-A\boldsymbol{c}^{k+1}\|^{2} \\ \notag
      =&G(\boldsymbol{c}^{k})-G(\boldsymbol{c}^{k+1})+(\boldsymbol{\gamma}^{k+1})^{T}A(\boldsymbol{c}^{k+1}-\boldsymbol{c}^{k})+\frac{\rho}{2}\|A(\boldsymbol{c}^{k+1}-\boldsymbol{c}^{k})\|^{2}.
    \end{align*}
Substituting (\ref{S-3}) into (\ref{S-2'}), we have that
    \begin{equation}\label{4.2}
        \boldsymbol{0}=\nabla G(\boldsymbol{c}^{k+1})-A\boldsymbol{\gamma}^{k+1}.
    \end{equation}
Since $G$ is convex and differentiable on $\mathbb{R}^{N}$, from the subgradient inequality, it follows that
    $$
       0\leq G(\boldsymbol{c}^{k})-G(\boldsymbol{c}^{k+1})-(\nabla G(\boldsymbol{c}^{k+1}))^{T}(\boldsymbol{c}^{k}-\boldsymbol{c}^{k+1})=G(\boldsymbol{c}^{k})-G(\boldsymbol{c}^{k+1})+(\boldsymbol{\gamma}^{k+1})^{T}A(\boldsymbol{c}^{k+1}-\boldsymbol{c}^{k}).
    $$
From three inequalites above, we have that
    \begin{equation}\label{4.3}
       \frac{\rho}{2}\|A(\boldsymbol{c}^{k+1}-\boldsymbol{c}^{k})\|^{2}\leq \mathcal{L}_{\rho}(\boldsymbol{\alpha}^{k+1},\boldsymbol{c}^{k},\boldsymbol{\gamma}^{k})-\mathcal{L}_{\rho}(\boldsymbol{\alpha}^{k+1},\boldsymbol{c}^{k+1},\boldsymbol{\gamma}^{k}).
    \end{equation}
Furthermore, from the definition of $\mathcal{L}_{\rho}$ and (\ref{S-3}) , it follows that
    \begin{equation*}
      -\frac{1}{\rho}\|\boldsymbol{\gamma}^{k+1}-\boldsymbol{\gamma}^{k}\|^{2}=\mathcal{L}_{\rho}(\boldsymbol{\alpha}^{k+1},\boldsymbol{c}^{k+1},\boldsymbol{\gamma}^{k})-\mathcal{L}_{\rho}(\boldsymbol{\alpha}^{k+1},\boldsymbol{c}^{k+1},\boldsymbol{\gamma}^{k+1}).
    \end{equation*}
By (\ref{S-3'}), it is easy to check that
    \begin{equation}\label{4.4}
      \|\boldsymbol{\gamma}^{k+1}-\boldsymbol{\gamma}^{k}\|=2\lambda\|\boldsymbol{c}^{k+1}-\boldsymbol{c}^{k}\|= 2\lambda\|A^{-1} A(\boldsymbol{c}^{k+1}-\boldsymbol{c}^{k})\|\leq 2\lambda\|A^{-1}\| \|A(\boldsymbol{c}^{k+1}-\boldsymbol{c}^{k})\|.
    \end{equation}
Combining two relations above, we conclude that
    \begin{equation}\label{4.5}
       -\frac{4\lambda^{2}\|A^{-1}\|^{2}}{\rho}\|A(\boldsymbol{c}^{k+1}-\boldsymbol{c}^{k})\|^{2}\leq  -\frac{1}{\rho}\|\boldsymbol{\gamma}^{k+1}-\boldsymbol{\gamma}^{k}\|^{2}=\mathcal{L}_{\rho}(\boldsymbol{\alpha}^{k+1},\boldsymbol{c}^{k+1},\boldsymbol{\gamma}^{k})-\mathcal{L}_{\rho}(\boldsymbol{\alpha}^{k+1},\boldsymbol{c}^{k+1},\boldsymbol{\gamma}^{k+1}).
    \end{equation}
Hence, (\ref{4.1}), (\ref{4.3}) and (\ref{4.5}) show that
    $$
       \left(\frac{\rho}{2}-\frac{4\lambda^{2}\|A^{-1}\|^{2}}{\rho}\right)\|A(\boldsymbol{c}^{k+1}-\boldsymbol{c}^{k})\|^{2}\leq \mathcal{L}_{\rho}(\boldsymbol{\alpha}^{k},\boldsymbol{c}^{k},\boldsymbol{\gamma}^{k})-\mathcal{L}_{\rho}(\boldsymbol{\alpha}^{k+1},\boldsymbol{c}^{k+1},\boldsymbol{\gamma}^{k+1}).
    $$
Since $\rho>4\lambda \|A^{-1}\|$, we have that
    $$
       \frac{\rho}{2}-\frac{4\lambda^{2}\|A^{-1}\|^{2}}{\rho}=\frac{\rho^{2}-8\lambda^{2}\|A^{-1}\|^{2}}{2\rho}>\frac{4\lambda^{2}\|A^{-1}\|^{2}}{\rho}>0.
    $$
Thus,
    $$
       \frac{4\lambda^{2}\|A^{-1}\|^{2}}{\rho}\|A(\boldsymbol{c}^{k+1}-\boldsymbol{c}^{k})\|^{2}\leq \mathcal{L}_{\rho}(\boldsymbol{\alpha}^{k},\boldsymbol{c}^{k},\boldsymbol{\gamma}^{k})-\mathcal{L}_{\rho}(\boldsymbol{\alpha}^{k+1},\boldsymbol{c}^{k+1},\boldsymbol{\gamma}^{k+1}).
    $$
Combining \cite[1.4.14 Proposition]{Megginson1998} with (\ref{3.2}), it shows that there exists $w_{1}>0$ such that
    $$
        w_{1}\|s^{k+1}-s^{k}\|_{\mathcal{H}}\leq \|A(\boldsymbol{c}^{k+1}-\boldsymbol{c}^{k})\|.
    $$
Let $\delta=\frac{4\lambda^{2}\|A^{-1}\|^{2}(w_{1})^{2}}{\rho}$. Thus $\delta>0$ and the following descent inequality holds
    $$
    \delta\|s^{k+1}-s^{k}\|_{\mathcal{H}}^{2}\leq \mathcal{L}_{\rho}(\boldsymbol{\alpha}^{k},\boldsymbol{c}^{k},\boldsymbol{\gamma}^{k})-\mathcal{L}_{\rho}(\boldsymbol{\alpha}^{k+1},\boldsymbol{c}^{k+1},\boldsymbol{\gamma}^{k+1}).
    $$
The proof is complete.
    \end{proof}

    Moreover, Lemma \ref{Lemma:4.2} shows that $\{\mathcal{L}_{\rho}(\boldsymbol{\alpha}^{k},\boldsymbol{c}^{k},\boldsymbol{\gamma}^{k})\}$ is monotonically decreasing, that is, for any $k\in \mathbb{N}$,
    \begin{equation}\label{4.6}
       \mathcal{L}_{\rho}(\boldsymbol{\alpha}^{1},\boldsymbol{c}^{1},\boldsymbol{\gamma}^{1})\geq \mathcal{L}_{\rho}(\boldsymbol{\alpha}^{k},\boldsymbol{c}^{k},\boldsymbol{\gamma}^{k})=F(\boldsymbol{\alpha}^{k})+G(\boldsymbol{c}^{k})-\frac{1}{2\rho}\| \boldsymbol{\gamma}^{k}\|^{2}+\frac{\rho}{2}\| \boldsymbol{\alpha}^{k}-A\boldsymbol{c}^{k}+\frac{1}{\rho}(\boldsymbol{\gamma}^{k})\|^{2}.
    \end{equation}
Moreover, since $A$ is symmetric and strictly positive definite, the minimum eigenvalue of $A$ is $\frac{1}{\|A^{-1}\|}$ which is its largest possible strong convexity parameter. Hence, \cite[Example 5.19 and Theorem 5.24 (iii)]{Beck2017} show that
    \begin{equation*}
       G(\boldsymbol{c}^{k})=\lambda (\boldsymbol{c}^{k}-\boldsymbol{0})^{T}(A\boldsymbol{c}^{k}-\boldsymbol{0})\geq \frac{\lambda}{\|A^{-1}\|}\|\boldsymbol{c}^{k}\|^{2}.
    \end{equation*}
Since $F(\boldsymbol{\alpha}^{k})\geq 0$ and $\rho>4\lambda \|A^{-1}\|$, the two inequalites above and (\ref{S-3'}) provide that
    \begin{equation}\label{4.7}
       \mathcal{L}_{\rho}(\boldsymbol{\alpha}^{1},\boldsymbol{c}^{1},\boldsymbol{\gamma}^{1})\geq G(\boldsymbol{c}^{k})-\frac{1}{2\rho}\| \boldsymbol{\gamma}^{k}\|^{2} \geq \frac{\lambda}{\|A^{-1}\|} \|\boldsymbol{c}^{k}\|^{2}-\frac{2\lambda}{\rho}\| \boldsymbol{c}^{k}\|^{2} \geq \frac{\lambda}{2\|A^{-1}\|} \| \boldsymbol{c}^{k}\|^{2}\geq 0.
    \end{equation}
It means that $\{\boldsymbol{c}^{k}\}$ is bounded which ensures $\{\boldsymbol{\gamma}^{k}\}$ is also bounded by (\ref{S-3'}). Furthermore, (\ref{S-3}) shows that
    \begin{equation*}
         \| \boldsymbol{\alpha}^{k}\| \leq \|A\boldsymbol{c}^{k}\|+\frac{1}{\rho}\|\boldsymbol{\gamma}^{k}-\boldsymbol{\gamma}^{k-1}\|\leq \|A\|\|\boldsymbol{c}^{k}\|+\frac{1}{\rho}(\|\boldsymbol{\gamma}^{k}\|+\|\boldsymbol{\gamma}^{k-1}\|),
    \end{equation*}
$\{\boldsymbol{\alpha}^{k}\}$ is bounded. In conclusion, $\{(\boldsymbol{\alpha}^{k},\boldsymbol{c}^{k},\boldsymbol{\gamma}^{k})\}$ is bounded. Let $S$ be the set of subsequential limits of $\{(\boldsymbol{\alpha}^{k},\boldsymbol{c}^{k},\boldsymbol{\gamma}^{k})\}$. \cite[Theorem 3.6 and 3.7]{Rudin1976Principles} show that $S$ is nonempty compact, and
    \begin{equation}\label{4.8}
       \lim_{k\rightarrow \infty} \mathrm{dist}((\boldsymbol{\alpha}^{k},\boldsymbol{c}^{k},\boldsymbol{\gamma}^{k}),S)=0.
    \end{equation}
where $\mathrm{dist}(\cdot,\cdot)$ denotes Euclidean distance.

    \begin{lemma}\label{Lemma:4.3}
       If the conditions in Theorem \text{\normalfont{\ref{Theorem:4.1}}} hold, then

       \text{\normalfont{(i)}} $\mathcal{L}_{\rho}$ is constant on $S$

       \text{\normalfont{(ii)}} $\{\mathcal{L}_{\rho}(\boldsymbol{\alpha}^{k},\boldsymbol{c}^{k},\boldsymbol{\gamma}^{k})\}$ converges to $\mathcal{L}_{\rho}(S)$.
    \end{lemma}

    \begin{proof}
    By Lemma \ref{Lemma:4.2}, (\ref{4.6}) and (\ref{4.7}), it follows that $\{\mathcal{L}_{\rho}(\boldsymbol{\alpha}^{k},\boldsymbol{c}^{k},\boldsymbol{\gamma}^{k})\}$ is monotonically decreasing and bounded. Hence, \cite[Theorem 3.24]{Rudin1976Principles} shows the convergence of $\{\mathcal{L}_{\rho}(\boldsymbol{\alpha}^{k},\boldsymbol{c}^{k},\boldsymbol{\gamma}^{k})\}$.

    For any $(\boldsymbol{\alpha}^{*},\boldsymbol{c}^{*},\boldsymbol{\gamma}^{*})\in S$, there exists a subsequence $\{(\boldsymbol{\alpha}^{k_{j}},\boldsymbol{c}^{k_{j}},\boldsymbol{\gamma}^{k_{j}})\}$ that converges to $(\boldsymbol{\alpha}^{*},\boldsymbol{c}^{*},\boldsymbol{\gamma}^{*})$. Since $F$ is L.S.C. and the other term of $\mathcal{L}_{\rho}$ is continuous, we have that $\mathcal{L}_{\rho}$ is lower-continuous. Hence, the lower semi-continuity of $\mathcal{L}_{\rho}$ at $(\boldsymbol{\alpha}^{*},\boldsymbol{c}^{*},\boldsymbol{\gamma}^{*})$ and the convergence of $\{\mathcal{L}_{\rho}(\boldsymbol{\alpha}^{k},\boldsymbol{c}^{k},\boldsymbol{\gamma}^{k})\}$ show that
    \begin{equation}\label{4.9}
       \mathcal{L}_{\rho}(\boldsymbol{\alpha}^{*},\boldsymbol{c}^{*},\boldsymbol{\gamma}^{*})\leq \liminf_{j\rightarrow \infty} \mathcal{L}_{\rho}(\boldsymbol{\alpha}^{k_{j}},\boldsymbol{c}^{k_{j}},\boldsymbol{\gamma}^{k_{j}})=\lim_{k\rightarrow \infty}\mathcal{L}_{\rho}(\boldsymbol{\alpha}^{k},\boldsymbol{c}^{k},\boldsymbol{\gamma}^{k}).
    \end{equation}

    Conversely, since $\boldsymbol{\alpha}^{k_{j}+1}$ is the minimizer of $\mathcal{L}_{\rho}(\boldsymbol{\alpha},\boldsymbol{c}^{k_{j}} ,\boldsymbol{\gamma}^{k_{j}})$, it shows that
    \begin{equation}\label{4.10}
       \mathcal{L}_{\rho}(\boldsymbol{\alpha}^{*},\boldsymbol{c}^{k_{j}},\boldsymbol{\gamma}^{k_{j}})\geq \mathcal{L}_{\rho}(\boldsymbol{\alpha}^{k_{j}+1},\boldsymbol{c}^{k_{j}},\boldsymbol{\gamma}^{k_{j}}).
    \end{equation}
From the continuity of $\mathcal{L}_{\rho}$ with respect to $\boldsymbol{c}$ and $\boldsymbol{\gamma}$, it holds
    \begin{equation}\label{4.11}
        \lim_{j\rightarrow \infty} \mathcal{L}_{\rho}(\boldsymbol{\alpha}^{*},\boldsymbol{c}^{k_{j}},\boldsymbol{\gamma}^{k_{j}})=\mathcal{L}_{\rho}(\boldsymbol{\alpha}^{*},\boldsymbol{c}^{*},\boldsymbol{\gamma}^{*}).
    \end{equation}
By (\ref{4.1}), (\ref{4.3}) and (\ref{4.5}), we have that
    \begin{equation*}
       \mathcal{L}_{\rho}(\boldsymbol{\alpha}^{k_{j}+1},\boldsymbol{c}^{k_{j}+1},\boldsymbol{\gamma}^{k_{j}+1})\leq \mathcal{L}_{\rho}(\boldsymbol{\alpha}^{k_{j}+1},\boldsymbol{c}^{k_{j}},\boldsymbol{\gamma}^{k_{j}})\leq \mathcal{L}_{\rho}(\boldsymbol{\alpha}^{k_{j}},\boldsymbol{c}^{k_{j}},\boldsymbol{\gamma}^{k_{j}}).
    \end{equation*}
Combining with two inequalities above, we verify that
     \begin{equation}\label{4.12}
    \lim\limits_{j\rightarrow \infty}\mathcal{L}_{\rho}(\boldsymbol{\alpha}^{k_{j}+1},\boldsymbol{c}^{k_{j}},\boldsymbol{\gamma}^{k_{j}})=\lim_{k\rightarrow \infty}\mathcal{L}_{\rho}(\boldsymbol{\alpha}^{k},\boldsymbol{c}^{k},\boldsymbol{\gamma}^{k}).
     \end{equation}
By \cite[Theorem 3.19]{Rudin1976Principles}, (\ref{4.10}), (\ref{4.11}) and (\ref{4.12}) show that
    \begin{equation}\label{4.13}
       \mathcal{L}_{\rho}(\boldsymbol{\alpha}^{*},\boldsymbol{c}^{*},\boldsymbol{\gamma}^{*})\geq \lim_{k\rightarrow \infty}\mathcal{L}_{\rho}(\boldsymbol{\alpha}^{k},\boldsymbol{c}^{k},\boldsymbol{\gamma}^{k}).
    \end{equation}
Finally, (\ref{4.9}) and (\ref{4.13}) assure that
    \begin{equation*}
       \mathcal{L}_{\rho}(\boldsymbol{\alpha}^{*},\boldsymbol{c}^{*},\boldsymbol{\gamma}^{*})=\lim_{k\rightarrow \infty}\mathcal{L}_{\rho}(\boldsymbol{\alpha}^{k},\boldsymbol{c}^{k},\boldsymbol{\gamma}^{k}).
    \end{equation*}
Hence, $\mathcal{L}_{\rho}$ is constant on $S$ and $\{\mathcal{L}_{\rho}(\boldsymbol{\alpha}^{k},\boldsymbol{c}^{k},\boldsymbol{\gamma}^{k})\}$ converges to $\mathcal{L}_{\rho}(S)$. The proof is complete.
    \end{proof}

    We are now ready for proving the main result of this section.

     \begin{proof}[Proof of Theorem \text{\normalfont{\ref{Theorem:4.1}}}]
     First, Lemma \ref{Lemma:4.2} and Lemma \ref{Lemma:4.3} (ii) show that for any $k$, $\mathcal{L}_{\rho}(\boldsymbol{\alpha}^{k},\boldsymbol{c}^{k},\boldsymbol{\gamma}^{k})\geq \mathcal{L}_{\rho}(S)$. We consider the following two cases:

    (I) If there exists an integer $k_{0}$ for which $\mathcal{L}_{\rho}(\boldsymbol{\alpha}^{k_{0}},\boldsymbol{c}^{k_{0}},\boldsymbol{\gamma}^{k_{0}})=\mathcal{L}_{\rho}(S)$, then for any $k>k_{0}+1$. Since $\{\mathcal{L}_{\rho}(\boldsymbol{\alpha}^{k},\boldsymbol{c}^{k},\boldsymbol{\gamma}^{k})\}$ is monotonically decreasing, we see that
    \begin{equation*}
        \mathcal{L}_{\rho}(\boldsymbol{\alpha}^{k+1},\boldsymbol{c}^{k+1},\boldsymbol{\gamma}^{k+1})=\mathcal{L}_{\rho}(\boldsymbol{\alpha}^{k},\boldsymbol{c}^{k},\boldsymbol{\gamma}^{k})=\mathcal{L}_{\rho}(S).
    \end{equation*}
Since $\delta>0$, for any $k>k_{0}+1$, Lemma \ref{Lemma:4.2} shows that $s^{k+1}=s^{k}$ which ensures that $\{s^{k}\}$ is convergent.

    (II) If $\mathcal{L}_{\rho}(\boldsymbol{\alpha}^{k},\boldsymbol{c}^{k},\boldsymbol{\gamma}^{k})>\mathcal{L}_{\rho}(S)$ for any integer $k$, then we verify the convergence of $\{s^{k}\}$ by KL property of $\mathcal{L}_{\rho}$. Combining with the nonnegativity and subanalyticity of $F$ and $G$, \cite[(I.2.1.9)]{1997Geometry} shows that $\mathcal{L}_{\rho}$ is subanalytic. Moreover, \cite{Bolte2007,Bolte2,Xu2013} assures that $\mathcal{L}_{\rho}$ is a KL function on $\mathbb{R}^{3m}$, that is, $\mathcal{L}_{\rho}$ has the KL property at each point in $\mathbb{R}^{3m}$ (see \cite[Section 2.4]{Bolte2014Proximal}). Clearly, by definition, $\mathcal{L}_{\rho}$ is a KL function on $S$. Hence, \cite[Lemma 3.6]{Bolte2014Proximal} assures that there exist $\varepsilon>0$, $\eta>0$ and a nonnegative continuous concave function $\varphi:[0,\eta)\rightarrow (0,+\infty)$ related to KL property such that

    (i) $\varphi(0)=0$ and $\varphi$ is continuously differentiable on $(0,\eta)$ with positive derivatives;

    (ii) if $\text{dist}((\boldsymbol{\alpha},\boldsymbol{c},\boldsymbol{\gamma}),S)<\varepsilon$ and $\mathcal{L}_{\rho}(S)<\mathcal{L}_{\rho}(\boldsymbol{\alpha},\boldsymbol{c},\boldsymbol{\gamma})<\mathcal{L}_{\rho}(S)+\eta$, then
    \begin{equation*}
       \varphi^{'}(\mathcal{L}_{\rho}(\boldsymbol{\alpha},\boldsymbol{c},\boldsymbol{\gamma})-\mathcal{L}_{\rho}(S))\ \text{\normalfont{dist}}(0,\partial \mathcal{L}_{\rho}(\boldsymbol{\alpha},\boldsymbol{c},\boldsymbol{\gamma})) \geq 1.
    \end{equation*}
where $\partial$ denotes the limiting subdifferential (see \cite[Definition 8.3]{Rockafellar1998}). From Lemma \ref{Lemma:4.3} (i) and (\ref{4.8}), it suffices to show that for $\varepsilon>0$ and $\eta>0$ above, there exists an integer $k_{1}$ such that for any $k>k_{1}$, we have
     \begin{equation}\label{4.14}
      \varphi^{'}(\mathcal{L}_{\rho}(\boldsymbol{\alpha}^{k},\boldsymbol{c}^{k},\boldsymbol{\gamma}^{k})-\mathcal{L}_{\rho}(S))\ \text{\normalfont{dist}}(0,\partial \mathcal{L}_{\rho}(\boldsymbol{\alpha}^{k},\boldsymbol{c}^{k},\boldsymbol{\gamma}^{k}))\geq 1.
    \end{equation}
Let $v^{k}=\mathcal{L}_{\rho}(\boldsymbol{\alpha}^{k},\boldsymbol{c}^{k},\boldsymbol{\gamma}^{k})-\mathcal{L}_{\rho}(S)>0$. From the concavity of $\varphi$, we get that
     \begin{equation*}
     \varphi^{'}(v^{k})(v^{k}-v^{k+1})\leq \varphi(v^{k})-\varphi(v^{k+1}).
    \end{equation*}
Multiplying $\text{\normalfont{dist}}(0,\partial \mathcal{L}_{\rho}(\boldsymbol{\alpha}^{k},\boldsymbol{c}^{k},\boldsymbol{\gamma}^{k}))$ on both side and using (\ref{4.14}), we obtain that
    \begin{equation}\label{4.15}
      v^{k}-v^{k+1} \leq \text{\normalfont{dist}}(0,\partial \mathcal{L}_{\rho}(\boldsymbol{\alpha}^{k},\boldsymbol{c}^{k},\boldsymbol{\gamma}^{k})) (\varphi(v^{k})-\varphi(v^{k+1})).
    \end{equation}
By \cite[8.8 Exercise (c) and 10.5 proposition]{Rockafellar1998}, it follows that
    \begin{equation*}
       \partial \mathcal{L}_{\rho}(\boldsymbol{\alpha}^{k},\boldsymbol{c}^{k},\boldsymbol{\gamma}^{k})=\partial_{\boldsymbol{\alpha}} \mathcal{L}_{\rho}(\boldsymbol{\alpha}^{k},\boldsymbol{c}^{k},\boldsymbol{\gamma}^{k})\times \nabla_{\boldsymbol{c}} \mathcal{L}_{\rho}(\boldsymbol{\alpha}^{k},\boldsymbol{c}^{k},\boldsymbol{\gamma}^{k})\times \nabla_{\boldsymbol{\gamma}} \mathcal{L}_{\rho}(\boldsymbol{\alpha}^{k},\boldsymbol{c}^{k},\boldsymbol{\gamma}^{k}),
    \end{equation*}
where
    \begin{align}\label{4.16}
         &\partial_{\boldsymbol{\alpha}} \mathcal{L}_{\rho}(\boldsymbol{\alpha}^{k},\boldsymbol{c}^{k},\boldsymbol{\gamma}^{k})=\partial F(\boldsymbol{\alpha}^{k})+\boldsymbol{\gamma}^{k}+\rho(\boldsymbol{\alpha}^{k}-A\boldsymbol{c}^{k}), \notag \\
         &\nabla_{\boldsymbol{c}} \mathcal{L}_{\rho}(\boldsymbol{\alpha}^{k},\boldsymbol{c}^{k},\boldsymbol{\gamma}^{k})=\nabla G(\boldsymbol{c}^{k}) -A\boldsymbol{\gamma}^{k}-\rho A(\boldsymbol{\alpha}^{k}-A\boldsymbol{c}^{k}),\\
         &\nabla_{\boldsymbol{\gamma}} \mathcal{L}_{\rho}(\boldsymbol{\alpha}^{k},\boldsymbol{c}^{k},\boldsymbol{\gamma}^{k})=\boldsymbol{\alpha}^{k}-A\boldsymbol{c}^{k} \notag.
    \end{align}
Invoking the optimality condition for (\ref{S-1'}) and (\ref{4.2}), we have that
    \begin{align}\label{4.17}
       -\rho \left(\boldsymbol{\alpha}^{k}-A\boldsymbol{c}^{k-1}+\frac{1}{\rho}\boldsymbol{\gamma}^{k-1}\right)&\in \partial F(\boldsymbol{\alpha}^{k}), \\
       A\boldsymbol{\gamma}^{k}&=\nabla G(\boldsymbol{c}^{k}). \notag
    \end{align}
From (\ref{4.16}), (\ref{4.17})  and (\ref{S-3}), we obtain further that
    \begin{align*}
         \boldsymbol{\alpha}^{\#}_{k}&:=\rho A(\boldsymbol{c}^{k-1}-\boldsymbol{c}^{k})+\boldsymbol{\gamma}^{k}-\boldsymbol{\gamma}^{k-1}\in \partial_{\boldsymbol{\alpha}} \mathcal{L}_{\rho}(\boldsymbol{\alpha}^{k},\boldsymbol{c}^{k},\boldsymbol{\gamma}^{k}), \notag \\
         \boldsymbol{c}^{\#}_{k}&:=A(\boldsymbol{\gamma}^{k-1}-\boldsymbol{\gamma}^{k})\in  \nabla_{\boldsymbol{c}} \mathcal{L}_{\rho}(\boldsymbol{\alpha}^{k},\boldsymbol{c}^{k},\boldsymbol{\gamma}^{k}),  \\
         \boldsymbol{\gamma}^{\#}_{k}&:=\frac{1}{\rho}(\boldsymbol{\gamma}^{k}-\boldsymbol{\gamma}^{k-1})\in \nabla_{\boldsymbol{\gamma}} \mathcal{L}_{\rho}(\boldsymbol{\alpha}^{k},\boldsymbol{c}^{k},\boldsymbol{\gamma}^{k}). \notag
    \end{align*}
Hence, $(\boldsymbol{\alpha}^{\#}_{k},\boldsymbol{c}^{\#}_{k},\boldsymbol{\gamma}^{\#}_{k})\in \partial \mathcal{L}_{\rho}(\boldsymbol{\alpha}^{k},\boldsymbol{c}^{k},\boldsymbol{\gamma}^{k})$, which ensures that
    \begin{equation}\label{4.18}
       \text{\normalfont{dist}}(0,\partial \mathcal{L}_{\rho}(\boldsymbol{\alpha}^{k},\boldsymbol{c}^{k},\boldsymbol{\gamma}^{k})) \leq \|(\boldsymbol{\alpha}^{\#}_{k},\boldsymbol{c}^{\#}_{k},\boldsymbol{\gamma}^{\#}_{k})\|\leq \| \boldsymbol{\alpha}^{\#}_{k} \|+\|\boldsymbol{c}^{\#}_{k}\|+\|\boldsymbol{\gamma}^{\#}_{k}\|.
    \end{equation}
From \cite[1.4.14 Proposition]{Megginson1998}, (\ref{4.4}) and (\ref{3.2}), we see that there exists $w_{2}>0$ such that
    $$
      \|\boldsymbol{\gamma}^{k}-\boldsymbol{\gamma}^{k-1}\|\leq 2\lambda \|A^{-1}\| \|A(\boldsymbol{c}^{k}-\boldsymbol{c}^{k-1})\|\leq 2\lambda w_{2}\|A^{-1}\| \|s^{k}-s^{k-1}\|_{\mathcal{H}}.
    $$
Inserting the inequality above into $\boldsymbol{\alpha}^{\#}_{k}$, $\boldsymbol{c}^{\#}_{k}$ and $\boldsymbol{\gamma}^{\#}_{k}$, we verify that
      \begin{align}\label{4.19}
         \| \boldsymbol{\alpha}^{\#}_{k} \| &\leq \left(\rho w_{2}+2\lambda w_{2}\|A^{-1}\|\right) \|s^{k}-s^{k-1}\|_{\mathcal{H}},\notag \\
         \|\boldsymbol{c}^{\#}_{k}\| &\leq 2\lambda w_{2}\|A^{-1}\|\|A\| \|s^{k}-s^{k-1}\|_{\mathcal{H}}, \\
         \|\boldsymbol{\gamma}^{\#}_{k}\| &\leq \frac{2\lambda w_{2}\|A^{-1}\|}{\rho}  \|s^{k}-s^{k-1}\|_{\mathcal{H}} \notag.
      \end{align}
Combining (\ref{4.18}) with (\ref{4.19}), we have that there exists $\zeta>0$ such that
      \begin{equation}\label{4.20}
         \text{\normalfont{dist}}(0,\partial \mathcal{L}_{\rho}(\boldsymbol{\alpha}^{k},\boldsymbol{c}^{k},\boldsymbol{\gamma}^{k})) \leq \zeta \|s^{k}-s^{k-1}\|_{\mathcal{H}}.
      \end{equation}
Since $\delta, \zeta>0$, by rearranging terms, whenever $k>k_{1}+1$, Lemma \ref{Lemma:4.2}, (\ref{4.15}) and (\ref{4.20}) assure that
    \begin{equation*}
     \|s^{k+1}-s^{k}\|_{\mathcal{H}} \leq \sqrt{\frac{\zeta}{\delta} \|s^{k}-s^{k-1}\|_{\mathcal{H}}(\varphi(v^{k})-\varphi(v^{k+1}))}\leq \frac{\|s^{k}-s^{k-1}\|_{\mathcal{H}}+\frac{\zeta}{\delta} (\varphi(v^{k})-\varphi(v^{k+1}))}{2}.
    \end{equation*}
By rearranging terms, we obtain further that
    $$
       \|s^{k+1}-s^{k}\|_{\mathcal{H}}\leq \|s^{k}-s^{k-1}\|_{\mathcal{H}}-\|s^{k+1}-s^{k}\|_{\mathcal{H}}+\frac{\zeta}{\delta} (\varphi(v^{k})-\varphi(v^{k+1})).
    $$
Summing up the above relation from $k=k_{1}+1,...,n$, since $\|s^{n+1}-s^{n}\|_{\mathcal{H}}>0$ and $\varphi(v^{n+1})>0$,, we see that
    \begin{align}\label{4.21}
      \sum_{k=k_{1}+1}^{n}\|s^{k+1}-s^{k}\|_{\mathcal{H}} &\leq \|s^{k_{1}+1}-s^{k_{1}}\|_{\mathcal{H}}-\|s^{n+1}-s^{n}\|_{\mathcal{H}}+ \frac{\zeta}{\delta}(\varphi(v^{k_{1}+1})-\varphi(v^{n+1})) \\
      & \leq \|s^{k_{1}+1}-s^{k_{1}}\|_{\mathcal{H}}+ \frac{\zeta}{\delta}\varphi(v^{k_{1}+1})<\infty \notag
    \end{align}
Hence, we denote $\Delta_{k}=\sum\limits_{j=k}^{\infty} \|s^{j+1}-s^{j}\|_{\mathcal{H}}$. Thus $\Delta_{k}<\infty$. For any $\varepsilon_{1}>0$, there exists an integer $k_{2}>k_{1}+1$ such that for any $j>k>k_{2}$,
    \begin{equation*}
       \|s^{j}-s^{k}\|_{\mathcal{H}} \leq \Delta_{k}-\Delta_{j}<\varepsilon_{1}.
    \end{equation*}
Therefore, $\{s^{k}\}$ is a Cauchy sequence. Since $\mathcal{H}$ is a Hilbert space which is a complete metric space, it means that $\{s^{k}\}$ is convergent.

    Combining (I) with (II), we conclude that $\{s^{k}\}$ converges to $s^{*}$. Moreover, the sequence $\{(\boldsymbol{\alpha}^{k},\boldsymbol{c}^{k},\boldsymbol{\gamma}^{k})\}$ converges to $(\boldsymbol{\alpha}^{*},\boldsymbol{c}^{*},\boldsymbol{\gamma}^{*})$ and $s^{*}=\sum\limits_{i=1}^{N}(\boldsymbol{c}^{*})_{i}K(\boldsymbol{x}_{i},\cdot)$. Next we verify that $s^{*}$ is a stationary point of Optimization (\ref{2.1}). First,  (\ref{S-3}) shows that
    \begin{equation}\label{4.22}
       \boldsymbol{0}=\boldsymbol{\alpha}^{*}-A\boldsymbol{c}^{*},
    \end{equation}
which ensures that
    \begin{equation*}
        \mathcal{L}_{\rho}(\boldsymbol{\alpha}^{*},\boldsymbol{c}^{*},\boldsymbol{\gamma}^{*})=F(\boldsymbol{\alpha}^{*})+G(\boldsymbol{c}^{*}).
    \end{equation*}
By the definition of $\mathcal{L}_{\rho}$, Lemma \ref{Lemma:4.3} (ii) and the continuity of $G$, we find that
    \begin{equation*}
      \lim_{k\rightarrow \infty}F(\boldsymbol{\alpha}^{k})=\lim_{k\rightarrow \infty} \mathcal{L}_{\rho}(\boldsymbol{\alpha}^{k},\boldsymbol{c}^{k},\boldsymbol{\gamma}^{k})-G(\boldsymbol{c}^{k})-(\boldsymbol{\gamma}^{k})^{T}(\boldsymbol{\alpha}^{k}-A\boldsymbol{c}^{k})-\frac{\rho}{2}\|\boldsymbol{\alpha}^{k}-A\boldsymbol{c}^{k}\|^2=F(\boldsymbol{\alpha}^{*}).
    \end{equation*}
By \cite[proposition 8.7]{Rockafellar1998} and (\ref{4.17}), passing to the limit above along $\{(\boldsymbol{\alpha}^{k},\boldsymbol{c}^{k},\boldsymbol{\gamma}^{k})\}$, it follows that
    \begin{equation}\label{4.23}
       -\boldsymbol{\gamma}^{*}\in \partial F(\boldsymbol{\alpha}^{*})=\partial F(A\boldsymbol{c}^{*}).
    \end{equation}
Thus, (\ref{3.2}) and (\ref{4.23}) ensure that
    \begin{equation}\label{4.24}
       -2\lambda s^{*}\in \partial \left(\frac{1}{N} \sum_{i=1}^{N}L(\boldsymbol{x}_{i},y_{i},\cdot)\right)(s^{*}).
    \end{equation}
On the other hand, we suppose that $\boldsymbol{c}^{*}=\boldsymbol{0}$. Thus (\ref{4.24}) and (\ref{S-3'}) assure that $\boldsymbol{\alpha}^{*}=\boldsymbol{0}$ and $\boldsymbol{\gamma}^{*}=\boldsymbol{0}$. Moreover, (\ref{4.24}) shows that
    \begin{equation}\label{4.25}
       \boldsymbol{0}\in \partial F(\boldsymbol{0}).
    \end{equation}
Since $F$ can be split of the variable into subvectors, Assumption \ref{Assumption:4.1} (ii) and \cite[proposition 10.5]{Rockafellar1998} and  \cite[D. Rescaling]{Rockafellar1998} assure that
    \begin{equation}\label{4.26}
       \boldsymbol{0}\notin \partial F(\boldsymbol{0})=\left(\frac{1}{N}\partial L(\boldsymbol{x}_{1},y_{1},0)\right)\times\cdots \times \left(\frac{1}{N}\partial L(\boldsymbol{x}_{N},y_{N},0)\right).
    \end{equation}
Clearly, (\ref{4.25}) and (\ref{4.26}) are contradiction. Hence, $\boldsymbol{c}^{*}\neq \boldsymbol{0}$, which ensures that $s^{*}\neq 0$. Thus, (\ref{3.1}) assures that
    \begin{equation}\label{4.27}
       \nabla(\lambda \|\cdot\|_{\mathcal{H}}^{2})(s^{*})=2\lambda\|s^{*}\|_{\mathcal{H}} \nabla(\|\cdot\|_{\mathcal{H}})(s^{*})=2\lambda\|s^{*}\|_{\mathcal{H}} \frac{s^{*}}{\|s^{*}\|_{\mathcal{H}}} =2\lambda s^{*}.
    \end{equation}
Therefore, \cite[Definition 1.77 and Proposition 1.107]{Mordukhovich2006}, (\ref{4.25}) and (\ref{4.28}) show that
    $$
       0\in \partial \left(\frac{1}{N} \sum_{i=1}^{N}L(\boldsymbol{x}_{i},y_{i},\cdot)+\lambda\|\cdot\|_{\mathcal{H}}^{2}\right)(s^{*})=\partial \left(\frac{1}{N} \sum_{i=1}^{N}L(\boldsymbol{x}_{i},y_{i},\cdot)\right)(s^{*})+\nabla(\lambda \|\cdot\|_{\mathcal{H}}^{2})(s^{*}),
    $$
that is, $s^{*}$ is a stationary point of Optimization (\ref{2.1}). The proof is complete.
   \end{proof}

    Next we analyze the convergent rates of $\{s^{k}\}$. By \cite[Example 5.3]{Bolte2014Proximal} and \cite[Theorem 3.1]{Bolte2007}, we know that $\varphi$ has the following form
    \begin{equation*}
       \varphi(z)=ez^{1-\theta}, \ \ \text{\normalfont{for}}\ e>0,\ \theta \in [0,1).
    \end{equation*}
Moreover, we have the following theorem about convergent rates.

    \begin{proposition}\label{Proposition:4.2}
     If the conditions in Theorem \text{\normalfont{\ref{Theorem:4.1}}} hold, then we have the following estimations:

     \text{\normalfont{(i)}} If $\theta=0$, then there exists an integer $n_{1}$ such that for $k>n_{1}$, $s^{k}=s^{*}$.

     \text{\normalfont{(ii)}} If $\theta \in (0,\dfrac{1}{2}]$, then there exists an integer $n_{2}$, $C_{1}>0$ and $\xi \in [0,1)$ such that for $k>n_{2}$,
     \begin{equation*}
         \| s^{k}-s^{*}\|_{\mathcal{H}}\leq C_{1}\xi^{k}.
     \end{equation*}

     \text{\normalfont{(iii)}} If $\theta \in (\dfrac{1}{2},1)$, then there exists an integer $n_{3}$ and $C_{2}>0$  such that for $k>n_{3}$,
     \begin{equation*}
        \| s^{k}-s^{*}\|_{\mathcal{H}}\leq C_{2} k^{\frac{1-\theta}{1-2\theta}}.
     \end{equation*}
    \end{proposition}

     \begin{proof}
    First we consider the case that $\theta=0$. Suppose that $\{\mathcal{L}_{\rho}(\boldsymbol{\alpha}^{k},\boldsymbol{c}^{k},\boldsymbol{\gamma}^{k})\}$ satisfying the case (II) in Theorem \ref{Theorem:4.1}. First, the definition of $\Delta_{k}$, we have that
    \begin{equation}\label{4.28}
         \|s^{k+1}-s^{k}\|_{\mathcal{H}}\to 0.
    \end{equation}
Moreover, (\ref{4.20}) and (\ref{4.28}) assure that when $k$ is sufficient largely,
    \begin{equation}\label{4.29}
        \text{\normalfont{dist}}(0,\partial \mathcal{L}_{\rho}(\boldsymbol{\alpha}^{k},\boldsymbol{c}^{k},\boldsymbol{\gamma}^{k}))\leq \zeta\|s^{k}-s^{k-1}\|_{\mathcal{H}}<\frac{1}{e}.
    \end{equation}
On the other hand, by the definition of $\varphi$, we have that $\varphi'(v^{k})=e$. Hence, (\ref{4.17}) shows that
    \begin{equation}\label{4.30}
        e \cdot \text{\normalfont{dist}}(0,\partial \mathcal{L}_{\rho}(\boldsymbol{\alpha}^{k},\boldsymbol{c}^{k},\boldsymbol{\gamma}^{k}))\geq 1.
    \end{equation}
Clearly, (\ref{4.29}) and (\ref{4.30}) are contradiction. Therefore, $\{\mathcal{L}_{\rho}(\boldsymbol{\alpha}^{k},\boldsymbol{c}^{k},\boldsymbol{\gamma}^{k})\}$ satisfies the case (I) in Theorem \ref{Theorem:4.1}, that is, there exists an integer $n_{1}$ such that whenever $k>n_{1}$, $s^{k}=s^{*}$. Item (i) holds.

    Next we consider the case that $\theta\in (0,1)$. It is easy to check that if $\{\mathcal{L}_{\rho}(\boldsymbol{\alpha}^{k},\boldsymbol{c}^{k},\boldsymbol{\gamma}^{k})\}$ satisfies the case (I) in Theorem 4.1, the assertion holds. If $\{\mathcal{L}_{\rho}(\boldsymbol{\alpha}^{k},\boldsymbol{c}^{k},\boldsymbol{\gamma}^{k})\}$ satisfies the case (II) in Theorem 4.1, then $\|s^{k+1}-s^{k}\|_{\mathcal{H}}>0$. Passing to the limit along the sequence $\{s^{k}\}$, the triangle inequality and (\ref{4.21}) ensure that for any $k>k_{1}$
    \begin{equation}\label{4.31}
      \|s^{k}-s^{*}\|_{\mathcal{H}} \leq \Delta_{k} \leq \Delta_{k-1}-\Delta_{k}+ \frac{\zeta}{\delta}\varphi(v^{k})=\Delta_{k-1}-\Delta_{k}+ \frac{\zeta}{\delta}e(v^{k})^{1-\theta}.
    \end{equation}
Since $\varphi'(v^{k})=\frac{e(1-\theta)}{(v^{k})^{\theta}}$, (\ref{4.15}) and (\ref{4.21}) assure that
    $$
        (v^{k})^{1-\theta}=((v^{k})^{\theta})^{\frac{1-\theta}{\theta}}\leq \left(e(1-\theta)\text{dist}(0,\partial \mathcal{L}_{\rho}(\boldsymbol{\alpha}^{k},\boldsymbol{c}^{k},\boldsymbol{\gamma}^{k}))\right)^{\frac{1-\theta}{\theta}} \leq \left(e(1-\theta)\zeta(\Delta_{k-1}-\Delta_{k})\right)^{\frac{1-\theta}{\theta}} .
    $$
Let $q=\frac{\zeta}{\delta}e[e(1-\theta)\zeta]^{\frac{1-\theta}{\theta}}$. Thus $q>0$. Combining with two inequalities above, we have that
    \begin{equation}\label{4.32}
      \Delta_{k} \leq \Delta_{k-1}-\Delta_{k}+ \frac{\zeta}{\delta}e[e(1-\theta)\zeta]^{\frac{1-\theta}{\theta}}(\Delta_{k-1}-\Delta_{k})^{\frac{1-\theta}{\theta}}=\Delta_{k-1}-\Delta_{k}+q(\Delta_{k-1}-\Delta_{k})^{\frac{1-\theta}{\theta}}.
    \end{equation}
Moreover, the definition of $\Delta_{k}$ ensures that there exists an integer $k_{3}$ such that for any $k>k_{3}$, we have that
    $$
       \Delta_{k-1}-\Delta_{k}<1.
    $$

    If $\theta \in (0,\frac{1}{2}]$, then $\frac{1-\theta}{\theta}\geq 1$. We denote $n_{2}=\max\{k_{1},k_{3}\}$. If $k>n_{2}$, then (\ref{4.32}) shows that
    \begin{equation}\label{4.33}
      \Delta_{k} \leq (1+q) (\Delta_{k-1}-\Delta_{k}).
    \end{equation}
This implies that $\Delta_{k} \leq \frac{q}{1+q} \Delta_{k-1}$. Let $C_{1}=(\frac{q}{1+q})^{-1-n_{2}} \Delta_{n_{2}}$ and $\xi=\frac{q}{1+q}$. Thus $C_{1}>0$, and $\xi \in [0,1)$. Combining (\ref{4.31}) with (\ref{4.33}), we show that
    $$
      \|s^{k}-s^{*}\|_{\mathcal{H}}\leq \Delta_{k} \leq \left(\frac{q}{1+q}\right)^{k-1-n_{2}}\Delta_{n_{2}}=C_{1}\xi^{k}.
    $$
Item (ii) holds.

    If $\theta \in (\frac{1}{2},1)$, then $\frac{1-\theta}{\theta}<1$. Whenever $k>n_{2}$, it follows that
    \begin{equation*}
       1\leq (1+q)^{\frac{\theta}{1-\theta}}(\Delta_{k-1}-\Delta_{k})\Delta_{k}^{-\frac{\theta}{1-\theta}}.
    \end{equation*}
Let $r\in (1,\infty)$. First we assume that ${\Delta_{k}^{-\frac{\theta}{1-\theta}}}\leq r{\Delta_{k-1}^{-\frac{\theta}{1-\theta}}}$, it holds that
    \begin{equation*}
        (\Delta_{k-1}-\Delta_{k})\Delta_{k}^{-\frac{\theta}{1-\theta}}\leq r(\Delta_{k-1}-\Delta_{k})\Delta_{k-1}^{-\frac{\theta}{1-\theta}}\leq r \int_{\Delta_{k}}^{\Delta_{k-1}}z^{-\frac{\theta}{1-\theta}}\text{d}z=\frac{1-\theta}{2\theta-1}r[\Delta_{k}^{\frac{1-2\theta}{1-\theta}}-\Delta_{k-1}^{\frac{1-2\theta}{1-\theta}}].
    \end{equation*}
Combining with two inequalities above, we find that
    \begin{equation*}
      \frac{2\theta-1}{(1-\theta)r}(1+q)^{-\frac{\theta}{1-\theta}}\leq \Delta_{k}^{\frac{1-2\theta}{1-\theta}}-\Delta_{k-1}^{\frac{1-2\theta}{1-\theta}}.
    \end{equation*}
Next we assume that ${\Delta_{k}^{-\frac{\theta}{1-\theta}}}>r{\Delta_{k-1}^{-\frac{\theta}{1-\theta}}}$. Since $-\frac{1-2\theta}{\theta}>0$, we have ${\Delta_{k}^{\frac{1-2\theta}{1-\theta}}}>r^{\frac{2\theta-1}{\theta}}{\Delta_{k-1}^{\frac{1-2\theta}{1-\theta}}}$. This ensures that
    \begin{equation*}
        (r^{\frac{2\theta-1}{\theta}}-1)\Delta_{n_{2}}^{\frac{1-2\theta}{1-\theta}}\leq (r^{\frac{2\theta-1}{\theta}}-1)\Delta_{k-1}^{\frac{1-2\theta}{1-\theta}}\leq \Delta_{k}^{\frac{1-2\theta}{1-\theta}}-\Delta_{k-1}^{\frac{1-2\theta}{1-\theta}}.
    \end{equation*}
Let $\mu=\min\left\{\frac{2\theta-1}{(1-\theta)r}(1+q)^{-\frac{\theta}{1-\theta}},(r^{\frac{2\theta-1}{\theta}}-1)\Delta_{n_{2}}^{\frac{1-2\theta}{1-\theta}}\right\}$. Thus $\mu>0$, and
    \begin{equation}\label{4.34}
      \mu \leq \Delta_{k}^{\frac{1-2\theta}{1-\theta}}-\Delta_{k-1}^{\frac{1-2\theta}{1-\theta}}.
    \end{equation}
Since $\frac{1-\theta}{1-2\theta}<0$, summing up the above relation from $j=n_{2},...,k-1$ and rearranging terms, (\ref{4.31}) and (\ref{4.34}) show that there exists $C_{2}>0$ such that if $k>n_{3}=n_{2}$,
    $$
       \|s^{k}-s^{*}\|_{\mathcal{H}} \leq \Delta_{k}\leq [\Delta_{n_{2}}^{\frac{1-2\theta}{1-\theta}}+(k-n_{2}+1)\mu]^{\frac{1-\theta}{1-2\theta}}\leq C_{2}k^{\frac{1-\theta}{1-2\theta}}.
    $$
Item (iii) follows immediately. The proof is complete.
     \end{proof}

   \section{Numerical Examples}
   \label{sec:5}
   In this section, we test Algorithm \ref{alg:ADMM} by the synthetic data and the real data. We choose some training data and testing data, loss functions and kernels to test Algorithm \ref{alg:ADMM}. For simplicity, let $L_{1}$, $L_{2}$, $L_{3}$ and $L_{4}$ be four loss functions used in our experiments, that is,
    $$
      L_{1}(\boldsymbol{x},y,t)= \begin{cases} 1-yt, & yt-1<0 \\ 0, & yt-1\geq 0 \end{cases},\
       \ L_{2}(\boldsymbol{x},y,t)= \begin{cases} -yt+2, & yt-1<-1  \\ -2yt+2, & -1\leq yt-1<0  \\ 0, & yt-1 \geq 0 \end{cases},
    $$
and
    $$
      L_{3}(\boldsymbol{x},y,t)= \begin{cases} log(2-yt), & yt-1<0 \\ 0, & yt-1 \geq 0  \end{cases},\ \ L_{4}(\boldsymbol{x},y,t)= \begin{cases} 1, & yt-1<-1 \\ 1-yt, & -1\leq yt-1<0 \\ 0, & yt-1 \geq 0  \end{cases}.
    $$
Here are the graphs of these support vector loss functions above.
    \begin{figure}[H]
    \centering
    \includegraphics[width=4.2 in]{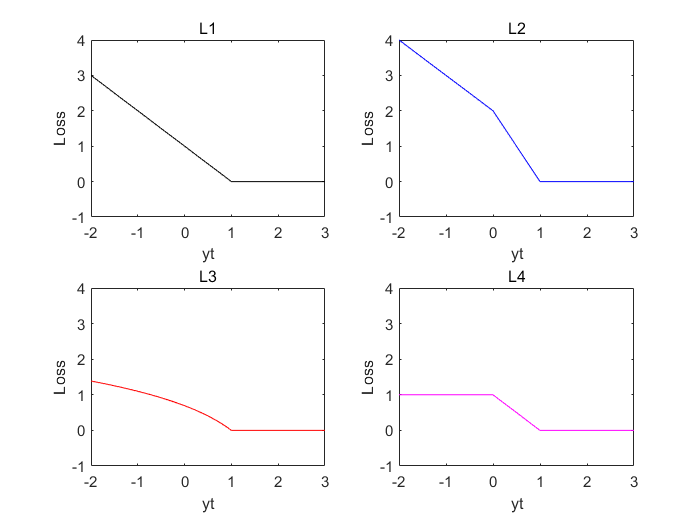}
    \caption{For the graphs of the loss functions $L_{1}$, $L_{2}$, $L_{3}$ and $L_{4}$, we replace $yt$ to $t$ because of the symmetry of $y=+1$ and $y=-1$.}
    \end{figure}

    We see that $L_{1}$ is convex Hinge loss, $L_{2}$ is a nonconvex linear piecewise loss function, $L_{3}$ is a nonconvex piecewise logarithmic loss function and $L_{4}$ is a nonconvex ramp loss function. These four loss functions satisfy Assumption \ref{Assumption:4.1}. On the other hand, let $K_{1}$ be Gaussian kernel, that is,
    \begin{equation*}
       K_{1}(\boldsymbol{x},\boldsymbol{x}')=\exp(-\sigma_{1} \| \boldsymbol{x}-\boldsymbol{x}'\|_{2}^{2}),\ \text{for}\ \sigma_{1}>0
    \end{equation*}
and $K_{2}$  be Mat\'ern 1-norm kernel, that is,
    \begin{equation*}
       K_{2}(\boldsymbol{x},\boldsymbol{x}')=\exp(-\sigma_{2} \| \boldsymbol{x}-\boldsymbol{x}' \|_{1}),\ \text{for} \ \sigma_{2}>0,
    \end{equation*}
where $\|\cdot \|_{1}$ denotes 1-norm in Euclidean space. Moreover, these two kernels are symmetric and strictly positive definite. Next we introduce our test results on synthetic data and real data.

\subsection{Examples on Synthetic Data}
    We sample from $\Omega_{1}=[-3,10]\times[-3,10]$ labeled by $+1$ and $\Omega_{2}=[-10,3]\times[-10,3]$ labeled by $-1$ randomly to obtain different training sets and testing sets. The data labeled by $+1$ are equal to the data labeled by $-1$ in each training set or testing set. Here is an example of sampling. In the following figures, two subdatasets are colored in blue and red.

    \begin{figure}[H]
	\centering
	\subfloat[Training Set]{
		\includegraphics[width=2.1 in]{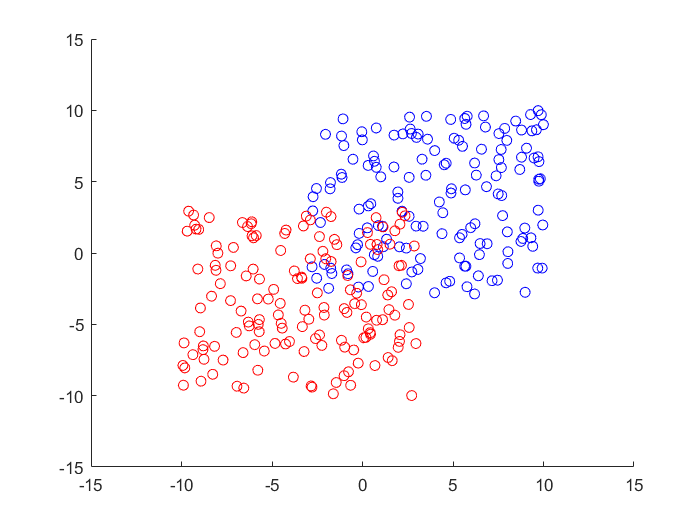}
		
    }
	\subfloat[Testing Set]{
		\includegraphics[width=2.1 in]{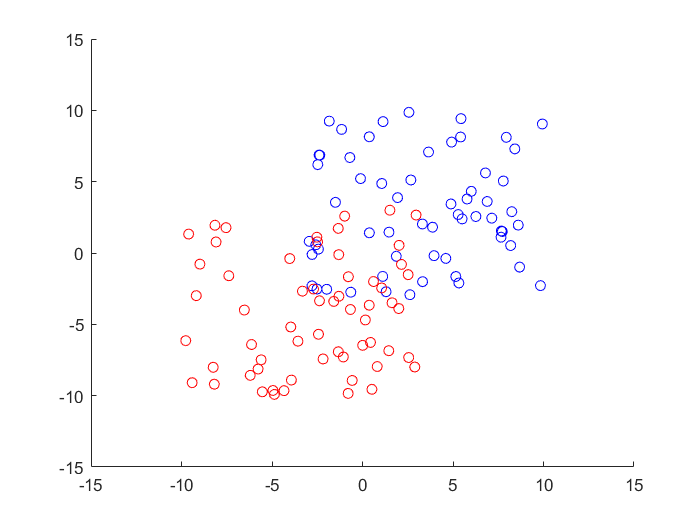}
		
	}
	\caption{Example of Sampling $(N=300)$.}
    \end{figure}

     First, we show the convergence of Algorithm \ref{alg:ADMM} for the nonconvex loss function $L_{3}$. Some parameters and results of the numerical experiment represent as follows.

     \begin{itemize}
      \setlength{\itemsep}{0pt}
      \setlength{\parskip}{0pt}
       \item Gaussian kernel $K_{1}$, where $\sigma_{1}=1$.
       \item The nonconvex loss function $L_{3}$.
       \item $N=300$.
       \item $\lambda=0.1$, $\rho=0.05$ and $\varepsilon_{0}=10^{-12}$.
       \item we choose initial values randomly in $[-10,10]^{N}$.
     \end{itemize}

     From (\ref{4.22}), we can show the convergence of $\{s^{k}\}$ by $\sum\limits_{k=1}^{\infty} \|s^{k+1}-s^{k}\|_{\mathcal{H}}$. By definition of $\{s^{k}\}$, (\ref{3.1}) and (\ref{3.2}), we have that
     $$
        \|s^{k+1}-s^{k}\|_{\mathcal{H}}=\sqrt{\sum_{i=1}^{N} (\boldsymbol{c}^{k+1}-\boldsymbol{c}^{k})_{i}(s^{k+1}-s^{k})(\boldsymbol{x}_{i})}=\sqrt{(\boldsymbol{c}^{k+1}-\boldsymbol{c}^{k})^{T}A(\boldsymbol{c}^{k+1}-\boldsymbol{c}^{k})}.
     $$
In the following two picture, we show the convergence of Algorithm 1 by $\sum\limits_{k=1}^{\infty} \|s^{k+1}-s^{k}\|_{\mathcal{H}}$.
     \begin{figure}[H]
	\centering

	\subfloat[$\|s^{k+1}-s^{k}\|_{\mathcal{H}}$]{
		\includegraphics[width=2.1 in]{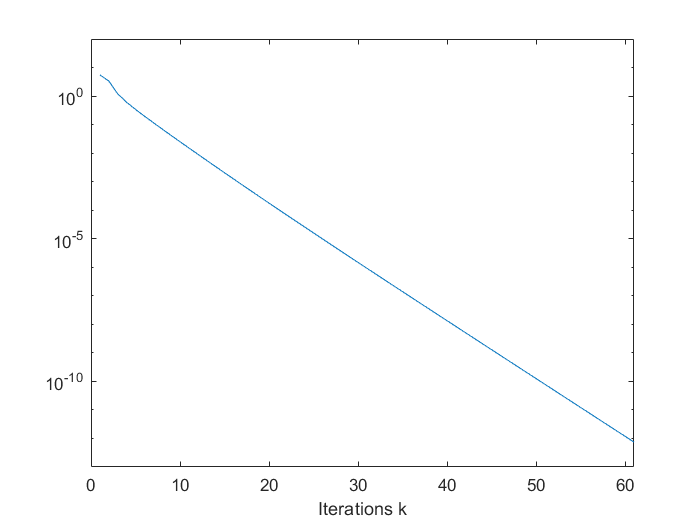}
		
    }
	\subfloat[$\sum_{k=1}^{\infty} \|s^{k+1}-s^{k}\|_{\mathcal{H}}$]{
		\includegraphics[width=2.1 in]{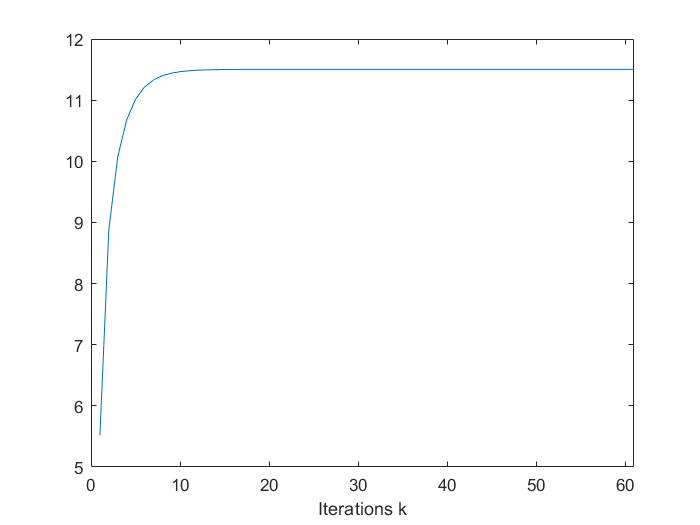}
		
	}
	\caption{Convergence of Algorithm 1 with Nonconvex Loss Function $L_{3}$.}
    \end{figure}

     Figure 3 shows that for the training data and parameters, Algorithm \ref{alg:ADMM} converges in 61 iterations. This shows the effectiveness of Algorithm \ref{alg:ADMM}. Next, we use different sizes of training set and testing set to test Algorithm \ref{alg:ADMM}. Some parameters and results of the numerical experiment represents as follows.

     \begin{itemize}
      \setlength{\itemsep}{0pt}
      \setlength{\parskip}{0pt}
       \item Gaussian kernel $K_{1}$, where $\sigma_{1}=1$.
       \item The nonconvex loss function $L_{2}$.
       \item $N=300$.
       \item $\lambda=0.1$, $\rho=1$ and $\varepsilon_{0}=10^{-12}$.
       \item we choose 20 initial values randomly in $[-10,10]^{N}$.
     \end{itemize}

    \begin{table}[H]
      \footnotesize
      \caption{Comparison of Different Sizes of Data.}
      \centering
      {\begin{tabular}{ccccc} \hline
                      Training Data  & Testing Data & Time(s) &  Training Accuracy &  Testing Accuracy \\ \hline
                       100 & 40 & 3.546 & 98\% & 85\% \\
                       200 & 80 & 7.302 & 94\% & 87.5\% \\
                       300 & 120 & 10.524 & 93.7\% & 90\% \\
                       400 & 160 & 14.833 & 93\% & 87.5\% \\
                       500 & 200 & 18.343 & 93.2\% & 90\% \\
                       600 & 240 & 21.813 & 92.7\% & 87.1\% \\
                       700 & 280 & 27.469 & 92.4\% & 89.3\% \\
                       800 & 320 & 33.776 & 92.1\% & 90.9\% \\
                       900 & 360 & 44.653 & 90.3\% & 90\% \\
                       1000 & 400 & 56.510 & 92.3\% & 89\% \\     \hline
       \end{tabular}}
     \end{table}

    Table 5.2 shows that solving Optimization (\ref{2.1}) by Algorithm \ref{alg:ADMM} is feasible in terms of running time and accuracy. Next we will show that choosing different kinds of loss functions and kernels have different accuracy.

    Now we sample from the area $\Omega_{1}$ and $\Omega_{2}$ randomly to obtain a training set with $300$ points and a testing set with $120$ points. Some parameters of these experiments represent as follows.

    \begin{itemize}
    \setlength{\itemsep}{0pt}
    \setlength{\parskip}{0pt}
      \item The kernels $K_{1}$ and $K_{2}$, where $\sigma_{1}=2$ and $\sigma_{2}=1$.
      \item The loss functions $L_{1}$, $L_{2}$, $L_{3}$ and $L_{4}$.
      \item $\lambda=0.5$, $\rho=5$ and $\varepsilon_{0}=10^{-12}$.
      \item Choose $20$ initial values randomly in $[-10,10]^{N}$.
    \end{itemize}

    In each experiment, we will choose a loss function and a kernel. The results of these experiments represent as follows.
    \begin{table}[h]
    \footnotesize
    \caption{Numerical Results of Using Different Loss Functions and Kernels.}
    \centering
   {\begin{tabular}{cccc} \hline
         Loss Function & Kernel & Training Accuracy & Testing Accuracy \\ \hline
        $L_{1}$   & $K_{1}$ &  95.7\% & $87.5\%$ \\
        $L_{2}$   & $K_{1}$ &  95\% & $88.3\%$ \\
        $L_{3}$   & $K_{1}$ &  95.7\% & $88.3\%$ \\
        $L_{4}$   & $K_{1}$ &  95\% & $87.5\%$ \\
        $L_{1}$   & $K_{2}$ &  92\% & $90.8\%$ \\
        $L_{2}$   & $K_{2}$ &  93.7\% & $90\%$ \\
        $L_{3}$   & $K_{2}$ &  94.7\% & $90.8\%$ \\
        $L_{4}$   & $K_{2}$ &  93\% & $90\%$ \\     \hline
    \end{tabular}}
    \label{table:2}
    \end{table}

    From Table \ref{table:2}, it is easy to see that the nonconvex loss function  $L_{3}$ performs better than $L_{1}$, $L_{2}$ and $L_{4}$ in these experiments. It shows that the SVM in the RKHS with the nonconvex loss function is better than the SVM in the RKHS with the convex loss functions for some cases. Next, we introduce the numerical experiment results on a real-world benchmark dataset.

   \subsection{Examples on UCI Machine Learning Repository}
   The vinho verde data in UCI machine learning repository has two kinds of wine samples. We will identify them based on physicochemical tests. There are 11 input variables about them, which are fixed acidity, volatile acidity, citric acid, residual sugar, chlorides, free sulfur dioxide, total sulfur dioxide, density, pH, sulphates and alcohol. We have 2000 wine samples in training set, and a half of them are labeled by +1 and the others are labeled by -1. Moreover, we have 718 wine samples in testing set, and a half of them are labeled by +1 and the others are labeled by -1. Next, we introduce some parameters of these experiments as follows.
    \begin{itemize}
    \setlength{\itemsep}{0pt}
    \setlength{\parskip}{0pt}
      \item The kernels $K_{1}$ and $K_{2}$, where $\sigma_{1}=5$ and $\sigma_{2}=5$.
      \item The loss functions $L_{1}$, $L_{2}$, $L_{3}$ and $L_{4}$.
      \item $\lambda=0.5$, $\rho=1$ and $\varepsilon_{0}=10^{-12}$.
      \item Choose $20$ initial values randomly in $[-10,10]^{N}$.
    \end{itemize}

In each experiment, we will choose a loss function and a kernel and we have the following results.
    \begin{table}[H]
    \footnotesize
    \caption{Numerical Results on Vinho Verde Data.}
    \centering
    \begin{tabular}{cccc} \hline
         Loss Function & Kernel & Training Accuracy & Testing Accuracy \\ \hline
        $L_{1}$   & $K_{1}$ &  99,9\% & $90.0\%$ \\
        $L_{2}$   & $K_{1}$ &  100\% & $90.8\%$ \\
        $L_{3}$   & $K_{1}$ &  99.9\% & $90.3\%$ \\
        $L_{4}$   & $K_{1}$ &  99.9\% & $90.3\%$ \\
        $L_{1}$   & $K_{2}$ &  100\% & $88.3\%$ \\
        $L_{2}$   & $K_{2}$ &  100\% & $91.8\%$ \\
        $L_{3}$   & $K_{2}$ &  100\% & $87.6\%$ \\
        $L_{4}$   & $K_{2}$ &  100\% & $91.2\%$ \\ \hline
    \end{tabular}
    \label{table:3}
    \end{table}

    From Table \ref{table:3}, we check that $L_{2}$ performs better than $L_{1}$, $L_{3}$ and $L_{4}$. It shows that in some cases the nonconvex loss function is more suitable than the convex loss function, which is our motivation of this paper.

   In Section \ref{sec:5}, we demonstrate the effectiveness of solving Optimization (\ref{2.1}) by Algorithm \ref{alg:ADMM}. In addition, we give some examples to show that in some cases, the SVM in the RKHS with the nonconvex loss functions are better than the SVM in the RKHS with the convex loss functions. Therefore, we should reconsider not only convex loss function but also nonconvex loss function. \\

\noindent {\bf Acknowledgments.} This program is supported by the National Natural Science Foundation of China under grant \#12071157, \#12026602 and the Natural Science Foundation of Guangdong \#2019A1515011995 and \#2020B1515310013.

\hbox to14cm{\hrulefill}\par
   \  \\

   {\small MINGYU MO} \par
   School of Mathematical Sciences \par
   South China Normal University \par
   Guangzhou, 510631, Guangdong, PR China \par
   Email-address: mmymaths@qq.com. \\
   \ \\

   {\small QI YE} \par
   School of Mathematical Sciences \par
   South China Normal University \par
   Guangzhou, 510631, Guangdong, PR China \\

   Pazhou Lab \par
   Guangzhou, 511442, Guangdong, PR China.\par
   Email-address: yeqi@m.scnu.edu.cn. \\

\end{document}